\documentclass{m2an}
\newcommand{\mycode}{\url{https://github.com/juliocareaga/wave-heat}}

\usepackage{amsfonts}
\usepackage{mathrsfs}
\usepackage[msc-links]{amsrefs}
\usepackage{scalerel,amssymb}
\usepackage{mathtools}

\usepackage{graphicx}
\usepackage{epstopdf}
\usepackage{tcolorbox}
\usepackage{tikz}
\usepackage{pgfplots}

\usepackage{algorithmic}
\usepackage{calligra}
\usepackage{autonum}
\usepackage{hhline}
\usepackage{array}
\usepackage{diagbox}
\usepackage{mdframed}
\usepackage{multicol}

\usepackage{subcaption}
\usepackage{moreverb}
\usepackage{bbm}
\usepackage{todonotes}
\usepackage{lineno}
\usepackage{todonotes}
\usepackage{siunitx}
\usepackage{subcaption}
\usepackage{alphalph}
\usepackage{booktabs}
\usepackage{makecell}
\usepackage{transparent}

\allowdisplaybreaks
\usetikzlibrary{arrows.meta}
\definecolor{mygreen}{HTML}{43a047}
\definecolor{darkgreen}{rgb}{0,0.5,0}

\def\HonetLtwo{H^1_t(\Ltwo)}

\def\solh{u_h}

\def\Hpoldegone{H^{\poldeg+1}(\Omega)}

\def\LinftHtwo{L^\infty_t(\Htwo)}

\def\Lthree{L^3(\Omega)}

\def\LtwoLthree{L^2(\Lthree)}
\def\LinftLsix{L^\infty_t(\Lsix)}
\def\LinftHone{L^\infty_t(\Hone)}
\def\Czerooneloc{C^{0,1}_{\textup{loc}}}
\def\Coneoneloc{C^{1,1}_{\textup{loc}}}
\newcommand{\dimension}{d}

\def\gh{g_h}

\def\eps{\varepsilon}

\newcommand{\Om}{\Omega}

\newcommand{\Dh}{\Delta_h}

\newcommand{\poldeg}{\eta}
\def\CLinf{C(\Linf)}

\newcommand{\Thetaa}{\Theta_{\textup{a}}}

\newcommand{\calQ}{{\mathcal{Q}}}

\newcommand{\norm}[1]{\| #1 \|}
\newcommand{\ip}[2]{ ( #1 , #2  )_{L^2} }

\newcommand{\finalth}{t^*_h}
\newcommand{\ph}{u_h}
\newcommand{\pt}{\partial_t}
\newcommand{\ptt}{u_{tt}}
\newcommand{\uh}{u_h}
\newcommand{\ut}{u_t}
\newcommand{\utt}{u_{tt}}

\def\Vh{V_h}

\newcommand{\LtwotLp}[1]{L^2_t(L^{#1}(\Omega))}
\newcommand{\LinftLp}[1]{L^\infty_t(L^{#1}(\Omega))}

\newcommand{\delt}{\partial_t}

\newcommand{\pht}{\partial_t u_{h}}

\newcommand{\uht}{\partial_t u_{h}}
\newcommand{\uhtt}{\partial_t^2 u_{h}}
\newcommand{\uth}{\partial_t u_{h}}

\newcommand{\phih}{\phi_h}
\def\tbeta{\tilde{\beta}}

\def\Linfty{L^\infty}

\def\calN{\mathcal{N}}
\def\kW{k_{\textup{W}}}
\def\kK{k_{\textup{K}}}

\newcommand{\Rhp}{\projRitz u}
\newcommand{\Rhpt}{\partial_t \projRitz u}
\newcommand{\Rhptt}{\partial^2_t \projRitz u}
\newcommand{\Rhu}{\projRitz u}
\newcommand{\Rhut}{\partial_t \projRitz u}
\newcommand{\Rhutt}{\partial^2_t \projRitz u}

\newcommand{\errhu}{e_h^u}
\newcommand{\errhut}{\partial_t e_h^u}
\newcommand{\errhutt}{\partial^2_t e_h^u}
\newcommand{\errhtheta}{e_h^\theta}
\newcommand{\errhthetat}{\partial_t e_h^\theta}

\def\Wpolponeinf{W^{\poldeg+1, \infty}(\Omega)}
\def\LtwotWdelta{L^2_t(\Wonedelta)}
\def\LinftWonedelta{L^\infty_t(\Wonedelta)}

\def\LinftWdelta{L^\infty_t(\Wonedelta)}

\def\wtdeltatheta{{\delta}^\theta}

\def\ehp{e_h^u}

\def\ehu{e_h^u}
\def\ehtu{\partial_t e_h^u}
\def\ehttu{\partial^2_t e_h^u}
\def\ehut{\partial_t e_h^u}
\def\ehutt{\partial^2_t e_h^u}
\def\ehtheta{e_h^\theta}
\def\ehttheta{\partial_t e_h^\theta}

\def\calFh{\mathcal{F}_h}

\def\calFhu{\calFh^u}
\def\calFhtheta{\calFh^\theta}
\def\wtdeltap{{\delta}^u}

\def\LinfLtwo{L^\infty(\Ltwo)}
\def\LinftLtwo{L^\infty_t(\Ltwo)}
\def\thetat{\theta_{t}}
\def\thetatt{\theta_{tt}}
\def\thetah{\theta_h}
\def\thetaht{\partial_t \thetah}

\def\calFhthetat{\calFh^{\thetat}}
\def\Rhtheta{\projRitz \theta}
\def\Rhthetat{\partial_t \projRitz \theta}
\def\Rhthetatt{\partial^2_t \projRitz \theta}

\def\LtwoLinf{L^2(\Linf)}
\def\LtwotLinf{L^2_t(\Linf)}
\def\thetazeroh{\theta_{0h}}

\def\fp{f}
\def\fph{f_h}
\def\fu{f}
\def\fuh{f_h}

\def\uzeroh{u_{0h}}
\def\uoneh{u_{1h}}
\def\Wonedelta{W^{1, \dimension+\delta}(\Omega)}
\def\ulq{\underline{q}}

\def\calI{\mathcal{I}}
\def\Xu{\mathcal{X}_u}
\def\Xtheta{\mathcal{X}_\theta}
\def\wh{w_h}
\def\LtwotWonedelta{L^2_t(\Wonedelta)}
\def\calR{\mathcal{R}}

\def\Deltah{\Delta_h}
\newcommand{\Triag}{\mathcal{T}_h}

\newcommand{\ds}{\, \textup{d} s }

\newcommand{\intt}{\int_0^t}

\newcommand{\R}{\mathbb{R}}
\newcommand{\N}{\mathbb{N}}
\newcommand{\Ltwo}{L^2(\Omega)}
\newcommand{\Lsix}{L^6(\Omega)}
\newcommand{\Woneinf}{W^{1,\infty}(\Omega)}
\newcommand{\Linf}{L^\infty(\Omega)}
\newcommand{\Lp}[1]{L^{#1}(\Omega)}
\newcommand{\Hone}{H^1(\Omega)}
\newcommand{\Htwo}{H^2(\Omega)}
\newcommand{\Honezero}{H_0^1(\Omega)}

\newcommand{\LtwoLtwo}{L^2(L^2(\Omega))}
\newcommand{\LtwotLtwo}{L^2_t(L^2(\Omega))}

\newcommand{\LinfLinf}{L^\infty(L^\infty(\Omega))}
\newcommand{\LinftLinf}{L^\infty_t(L^\infty(\Omega))}

\newcommand{\projRitz}{\textup{R}_h}
\newcommand{\Ih}{\textup{I}_h}
\newcommand{\Id}{\textup{I}}
\newtheorem{theorem}{Theorem}[section]% meant for sectionwise numbers
\newtheorem{corollary}[theorem]{Corollary}
\newtheorem{lemma}[theorem]{Lemma}% meant for sectionwise numbers
\newtheorem{remark}[theorem]{Remark}% meant for sectionwise numbers
\newtheorem{proposition}[theorem]{Proposition}% meant for sectionwise numbers
\makeatletter
\newcommand{\leqnomode}{\tagsleft@true}
\newcommand{\reqnomode}{\tagsleft@false}
\makeatother
\definecolor{grey}{rgb}{0.5,0.5,0.5}

\def\Gronwall{Gr\"onwall}

\usepackage{marginnote}

\usepackage[T1]{fontenc}    % proper hyphenation of umlauts
\usepackage[utf8]{inputenc} % if using pdfLaTeX (not needed with XeLaTeX/LuaLaTeX)
\usepackage[english]{babel}

\makeatletter
\def\@cite#1#2{{[#1]\if@tempswa\typeout
		{WSPC warning: optional citation argument ignored: `#2'} \fi}}

\def\@biblabel#1{[#1]}
\makeatletter

\numberwithin{equation}{section}

\begin{document}
%%-----------------------------
%%      the top matter
%%-----------------------------

%%-----------------------------
%%      the top matter
%%-----------------------------
\title{Finite element discretization\\[1mm] of nonlinear models of ultrasound heating}
\author{Julio Careaga}
\address{Bernoulli Institute, University of Groningen, Nijenborgh 9, 9747 AG Groningen, The Netherlands,\\      j.c.careaga.solis@rug.nl}
\author{Benjamin D\"orich}
\address{Institute for Applied and Numerical Mathematics, Karlsruhe Institute of Technology,\\  Englerstra{\ss}e 2, 76149 Karlsruhe, Germany,\\ benjamin.doerich@kit.edu}
\author{Vanja Nikoli\'c}
\address{Department of Mathematics, Radboud University, Heyendaalseweg 135, 6525 AJ Nijmegen, The Netherlands, \\ vanja.nikolic@ru.nl}
% %
%
% \date{...}
%
\begin{abstract}
	Heating generated by high-intensity focused ultrasound waves is central to many emerging medical applications, including non-invasive cancer therapy and targeted drug delivery. In this study, we aim to gain a fundamental understanding of numerical simulations in this context by analyzing conforming finite element approximations of the underlying nonlinear models that describe ultrasound-heat interactions. These models are based on a coupling of a nonlinear Westervelt--Kuznetsov acoustic wave equation to the heat equation with a pressure-dependent source term. A particular challenging feature of the system is that the acoustic medium parameters may depend on the temperature. The core of our new arguments in the \emph{a priori} error analysis lies in devising energy estimates for the coupled semi-discrete system that can accommodate the nonlinearities present in the model. To derive them, we exploit the parabolic nature of the system thanks to the strong damping present in the acoustic component. Theoretically obtained optimal convergence rates in the energy norm are confirmed by the numerical experiments. In addition, we conduct a further numerical study of the problem, where we simulate the propagation of acoustic waves in liver tissue for an initially excited profile and under high-frequency sources.
\end{abstract}
%
%\begin{resume} ... \end{resume}
%
\maketitle

\section{Introduction} \label{Sec: Introduction}

High-intensity focused ultrasound (HIFU) waves are known to act as a source of heat within the body. This heating phenomenon is at the core of many developing medical applications, including non-invasive ablation of cancer and targeted drug delivery; see, e.g.,~\cite{ter2016hifu, ellens2023high} for details. Rigorous mathematical research into the underlying (inherently nonlinear) models of wave-heat interactions has been initiated relatively recently with the contributions of, e.g.,~\cite{wilke2023p, nikolic2022local, nikolic2022westervelt}, which have investigated local and global well-posedness of the exact models. To the best of our knowledge, rigorous numerical understanding in this context is currently missing in the literature. In this work, we investigate conforming finite element approximations of the underlying models and develop a theoretical framework for their \emph{a priori} error analysis.\\
\indent Ultrasound-heat interactions present in HIFU-induced heating can be captured using a coupled system based on a damped nonlinear acoustic equation for $u$, representing either the acoustic pressure or acoustic velocity potential and the heat equation given in terms of the fluctuations $\theta$ of the ambient temperature as follows:
\begin{equation} \label{coupled_problem}
	\left\{ \begin{aligned}
		&\utt-{q(\theta)}\Delta u - \beta(\theta) \Delta \ut +  \calN(u, \ut, \utt, \nabla u, \nabla \ut, \theta)= \fu,  &&\text{in} \ \Omega \times (0,T), \\[1mm]
		& 	\theta_t -\kappa\Delta \theta+ \nu\theta = \calQ(u, \ut, \theta),  &&\text{in} \ \Omega \times (0,T),
	\end{aligned} \right.
\end{equation}
with damping coefficients $\kappa$, $\nu >0$.  Heating occurs due to the acoustic energy that is absorbed by the tissue, which is here modeled by having a pressure-dependent source term $\calQ$ in the heat equation. The acoustic medium parameters are known to depend on the temperature, resulting in the so-called \emph{thermal lensing} effect, where the focal region of the ultrasound waves may shift with changes in the temperature. In particular, this temperature dependency is seen in the {speed of sound squared $q$} and sound diffusivity $\beta$; the latter is computed using the relation
\begin{equation} \label{damping coeff}
	\begin{aligned}
	%	\beta = 2\frac{\tilde{\alpha} c^3}{\omega^2},
		\beta \coloneqq 2\frac{\tilde{\alpha} {q^{3/2}}}{\omega^2},
	\end{aligned}
\end{equation}
where $\tilde{\alpha}\coloneqq \tilde{\alpha}(\theta)$ denotes the acoustic amplitude absorption coefficient and $\omega$ the angular frequency; see~\cite{connor2002bio}. \\
\indent  The acoustic wave equation in  \eqref{coupled_problem} generalizes the classical Westervelt and Kuznetsov equations in nonlinear acoustics. In the case of the Westervelt equation, $u$ in \eqref{coupled_problem} represents the acoustic pressure $p$, whereas in the Kuznetsov equation $u$ represents the acoustic velocity potential $\psi$. The two quantities can be related using $p=\rho \psi_t$, where $\rho$ is the medium density.  Concerning the nonlinearities, these two classical equations are recovered with the following choices:
\begin{equation}
	\begin{aligned}
		\calN = \begin{cases}
			\,	\kW(\theta)\left(u^2\right)_{tt}= 2\kW(\theta)\left(u \utt+\ut^2 \right)  &\textup{Westervelt's equation}, \\[1mm]
			\,	\kK(\theta)\left(\ut^2\right)_{t} +(|\nabla u|^2)_t =2\kK(\theta)\ut \utt +2\nabla u \cdot \nabla \ut  \quad  &\textup{Kuznetsov's equation},
		\end{cases}
	\end{aligned}
\end{equation}
where the temperature-dependent nonlinearity coefficients are given by
\begin{equation} \label{nonlinearity coeffs}
	\kW \coloneqq \frac{1}{\rho {q}}\left(1+\frac{B}{2A}\right), \ \quad \kK \coloneqq \frac{1}{{q}}\frac{B}{2A}.
\end{equation}
In \eqref{nonlinearity coeffs}, $\frac{B}{A}$ represents the acoustic nonlinearity parameter of the medium, {where $A$ and $B$ arise as coefficients in the Taylor expansion of the pressure-density relation around the ambient values}. Note that the acoustic nonlinearity parameter is also known to depend on the temperature; see, for example,~\cite[Fig.\ 7]{van2011feasibility}.

\subsection{Mathematical generalization of the model} \label{Sec: Assumptions functions}To encompass both nonlinearity cases we assume in the analysis that the functional $\calN$ is given by
\begin{equation} \label{def calN}
	\calN(u, \ut, \utt, \nabla u, \nabla \ut, \theta)= \kW(\theta)\left(u^2\right)_{tt} + \kK(\theta)\left(\ut^2\right)_{t} +\ell(|\nabla u|^2)_t, \ \ell \in \R.
\end{equation}
Concerning the nonlinearity coefficients in \eqref{def calN}, we assume that
\begin{equation}
	\kW,\, \kK \in \Czerooneloc(\R).
\end{equation}
\subsubsection*{Assumptions on the medium parameters} 
%Let $q =c^2$. 
%{To simplify the notation, we use from now on the letter $q$ to represent the speed of sound squared ($q(\theta)= (c(\theta)^2)$), and assume its polynomial dependence on the temperature. More precisely,} 
Regarding the temperature-dependent speed of sound squared $q$ and sound diffusivity $\beta$, 
we assume that  $q \in P_m(\R)$ and {$\beta \in P_n(\R)$} are polynomials over $\R$ of maximal degree $m \in \N$ and $n \in \N$, respectively, and
such that
\[
q_0:=q(0)>0, \quad  \beta_0  := \beta(0)> 0.
\]
These positivity assumptions correspond to the usual assumptions of positivity of the speed of sound and sound diffusivity at constant temperatures.  We emphasize that the condition $\beta_0>0$ is particularly important as the presence of strong damping in the acoustic component (that is, having $-\beta_0 \Delta \ut$) will allow us to employ parabolic estimates in the numerical analysis.\\
\indent In practice,  the speed of sound and the acoustic attenuation coefficient  are indeed typically determined via a least-squares fit from data assuming polynomial dependence on the temperature; see, e.g.,~\cite{bilaniuk1993speed, connor2002bio, hallaj2001simulations}.

\subsubsection*{Assumptions on the absorbed energy}  In the literature, different forms of the functional $\calQ$ in \eqref{coupled_problem} are employed; see, e.g.,~\cite{shevchenko2012multi, hallaj2001simulations, connor2002bio}, and the references contained therein. We assume here that $\calQ$ has the following form:
\begin{equation}\label{Q_Form}
	\calQ (u, \ut, \theta) \coloneqq  \alpha(\theta) (\zeta_1 u^2+\zeta_2 \ut^2) \quad \text{with }  \alpha \in \Coneoneloc(\R), \quad \zeta_1, \zeta_2 \in \R.
\end{equation}
This allows us to cover, for example, the plane wave approximation for the volume rate of heat deposition (see~\cite[eq.\ (10.2.11)]{pierce2019acoustics})  given by $\calQ = \frac{ \tilde{\alpha}}{\rho {q^{1/2}}} p^2$ in both Westervelt and Kuznetsov regimes, where the acoustic pressure is $p =u$ and $p= \rho \ut$, respectively. Another expression for the absorbed energy found in the literature (see, e.g.,~\cite{nikolic2022local}) is $\calQ = \frac{2\beta}{\rho {q^2}}p_t^2$,
which is covered here in the Westervelt regime, where $p=u$. \\
\indent {The assumption  that  $\alpha$ belongs to $\Coneoneloc(\R)$ is needed in the error analysis of the heat subproblem in Section~\ref{semi-discret problem heat}; see Lemmas~\ref{lemma: est defect theta} and~\ref{lemma: est calFhtheta}. }
\begin{remark}
	Note that according to the available well-posedness results for the Westervelt--heat systems, the non-degeneracy condition
	\begin{equation}
		q(\theta) \geq \ulq >0
	\end{equation}
	is expected to hold for sufficiently smooth and small pressure-temperature data. Thus, in the case	$\alpha \sim \beta/q^2$, the assumed regularity of $\alpha$
	follows by the assumed properties of $q$ and $\beta$.
\end{remark}

\subsection*{Notation} We use $x \lesssim y$ below to denote $x \leq C y$, where $C>0$ does not depend on the spatial discretization parameter $h$. By $(\cdot, \cdot)_{L^2}$ we denote the scalar product on $\Ltwo$. We often omit the temporal domain $(0,T)$ when denoting the norms in Bochner spaces; for example, $\|\cdot\|_{L^p(L^q(\Omega))}$ denotes the norm on $L^p(0,T; L^q(\Omega))$. We use the subscript $t$ to emphasize that the temporal domain is $(0,t)$ for some $t \in (0,T)$; for example, $\|\cdot\|_{L^p_t(L^q(\Omega))}$ denotes the norm on $L^p(0,t; L^q(\Omega))$ for $t \in (0,T)$.

	\subsection{Assumptions on the exact solution} \label{Sec: Assumptions exact solution} Let $\Omega \subset \R^\dimension$, $\dimension \in \{1,2,3\}$, be an open and bounded set. Considering  data, we assume homogeneous Dirichlet boundary conditions for the pressure and temperature and sufficiently regular initial pressure and temperature data. That is, we consider the approximation of the following initial-boundary value problem:
\begin{equation}  \label{coupled_ibvp}
	\left\{ \begin{aligned}
		& \utt-q(\theta)\Delta u - \beta(\theta) \Delta \ut +\calN(u, \ut, \utt, \nabla u, \nabla \ut, \theta)= \fu \ \, &&\text{in} \ \Omega \times (0,T), \\[1mm]
		& 	\theta_t -\kappa\Delta \theta+ \nu\theta = \calQ(u, \ut, \theta) \qquad &&\text{in} \ \Omega \times (0,T), \\[1mm]
		& 	u_{\vert \partial \Om}=	\theta_{\vert \partial \Om}=0  \qquad &&{\text{in} \ (0,T)}, \\[1mm]
	%	& 	(u, \ut)_{\vert t=0}= (u_0, u_1), \quad \theta_{\vert t=0}=\theta_0, \\[1mm]
		& 
		{u_{\vert t=0} = u_0, \quad
		{\ut}_{\vert t=0}= u_1, \quad \theta_{\vert t=0}=\theta_0}  \qquad &&\text{in} \ \Omega,
	\end{aligned} \right.
\end{equation}
with $\calN$ as in \eqref{def calN}. Given $\eta \geq 1$ (which will denote the polynomial degree of the finite element basis functions on an element), we assume that there exists a unique solution of the problem such that
\begin{equation}
	\begin{aligned}
		(u, \theta) \in \Xu \times \Xtheta,
	\end{aligned}
\end{equation}
with
\begin{equation}
	\|u\|_{\Xu} + \|\theta\|_{\Xtheta} \leq C
\end{equation}
for some $C>0$, where the two spaces are defined as follows:
\begin{equation}
	\begin{aligned}
		\Xu \coloneqq \Bigl\{u:&\  u \in L^\infty\left(0,T; \Hpoldegone \cap \Woneinf \cap \Honezero\right),\\&\  \ut \in L^\infty\left(0,T; \Hpoldegone\cap \Woneinf \cap \Honezero\right) \\
		&\ \utt \in L^2(0,T; \Hpoldegone) \Bigr\}
	\end{aligned}
\end{equation}
and
\begin{equation}
	\begin{aligned}
		\Xtheta \coloneqq \Bigl\{\theta:&\ \theta \in                                                                                                                                                                                                                                                                                                                                                                                                                                                                                                                                                                                                                                                                                                                                                                                                                                                                                                                                                                                                                                                                                                                                                                                                                                                                                                                                                L^\infty\left(0,T; \Hpoldegone \cap W^{1, \infty}(\Omega) \cap \Honezero\right)
%		\\& \hspace*{1cm} 
		\cap L^2\left(0,T;  W^{\poldeg+1, \dimension+\delta }(\Omega)\cap W^{\poldeg +1 , \infty}(\Omega)\right), \\
		&\ \thetat \in L^\infty(0,T; \Hpoldegone \cap W^{1, \infty}(\Omega){)} \Bigr\}
	\end{aligned}
\end{equation}
with $\dimension+\delta  \in [2,6]$. For the upcoming analysis, it is worth noting that
\[
\Xu, \, \Xtheta \hookrightarrow L^\infty(0,T; \Linf).
\]
Furthermore, our main theoretical result (see Theorem~\ref{thm main}) assumes that there exists a sufficiently small $r>0$, such that
\begin{equation}  \label{smallness condition u theta}
	\begin{aligned}
		\begin{multlined}[t]
			\norm{{ \beta(\theta)-\beta_0}}_{\LinfLinf}
			+ \|\kW(\theta) u\|_{\CLinf}
			+ \|\kK (\theta) \ut\|_{\CLinf}
			+ \|\alpha(\theta) \ut\|_{\CLinf} \leq r.
		\end{multlined}
	\end{aligned}
\end{equation}
%~\\
 {Note that the smallness of $\beta(\theta)-\beta_0 \equiv \beta(\theta)-\beta(0)$ means that the fluctuation of the ambient temperature is required to be small.}  The small-data well-posedness analysis of~\eqref{coupled_ibvp} in the Westervelt case (and somewhat simplified function $\calQ$) based on energy arguments can be found in~\cite{nikolic2022local, nikolic2022westervelt}, under the assumption that the function $q$ does not degenerate and that $\beta=const.>0$. In~\cite{wilke2023p}, the concept of maximal $L^p$-$L^q$ regularity has been utilized to show local and global well-posedness of the non-isothermal Westervelt equation. The small-data local and global well-posedness of the Kuznetsov equation with constant medium parameters can be found in~\cite{mizohata1993global}. Although wave-heat system \eqref{coupled_ibvp} in its full generality assumed here does not appear to have been studied rigorously in the literature in terms of well-posedness, we expect that the general framework of~\cite{wilke2023p} can be utilized for this purpose. The smallness assumption in \eqref{smallness condition u theta} can then be enforced (via continuous dependence on data) by the smallness of initial data $(u_0, u_1, \theta_0)$.

\subsection{Main result} We next present the main theoretical result of this work.  We employ Lagrange finite elements here on a quasi-uniform triangulation  $\Triag$ and introduce the finite element space
incorporating the homogeneous boundary conditions
\begin{equation}
	\begin{aligned} \label{def:Vh}
		\Vh \coloneqq \Big\{ \phih \in C(\Om) \mid 
		  \phih |_{\partial \Omega} = 0
		 \, \, \text{and}
		\, \,	
			\phih |_K \in \mathcal{P}_\poldeg(K)   \text{ for all } K \in \Triag  \Big\}
	\end{aligned}
\end{equation}
of piecewise polynomials of degree $\poldeg \geq 1$, which is used both for the pressure and temperature. We introduce the Ritz
projection $\projRitz \colon \Honezero \to \Vh$
defined for $\varphi \in \Honezero$
via
\begin{align}
	\ip{\nabla \varphi}{ \nabla \phih} &= \ip{\nabla \projRitz \varphi}{ \nabla \phih}
\end{align}
for all $\phih \in \Vh$.  Further, we rely on the nodal interpolation operator $\Ih \colon C(\Omega) \to \Vh$, and define the discrete Laplacian operator $\Deltah \colon \Vh \to \Vh$
for $\psi_h,\phih \in \Vh$
via the relation
\begin{equation}
	\ip{\Deltah \psi_h}{\phih} \coloneqq - \ip{\nabla \psi_h}{ \nabla \phih}.
\end{equation}
Further, we introduce the bilinear functional $a(\cdot, \cdot): \Vh \times \Vh \rightarrow \R$ as follows:
\begin{equation}
	a(\psi_h, \phi_h) \coloneqq (\nabla \psi_h,  \nabla \phi_h)_{L^2}.
\end{equation}
With these preparations, we consider the semi-discrete acoustic problem:
\begin{subequations} \label{semi-discret problem wave}
	\begin{equation+} \label{discrete wave}
		\begin{aligned}
			\begin{multlined}[t](\uhtt, \phih)_{L^2} +a(\uh, q(\thetah) \phih)
				+
				a(\pht, \beta (\thetah) \phih) \\\hspace*{1.5cm}+(\calN(\uh, \uht, \uhtt, \nabla \uh, \nabla \uht, \thetah), \phih)_{L^2}
				= (\fph, \phih)_{L^2}
			\end{multlined}
		\end{aligned}
	\end{equation+}
	for all $\phih \in \Vh$, $t \in [0,T]$, with
	\begin{equation+} \label{discrete wave IC}
		(\uh, \uht)_{\vert t=0} = (\uzeroh,\uoneh).
	\end{equation+}
\end{subequations}
Note that, with {$\tbeta(\thetah) \coloneqq \beta(\thetah)-\beta_0$ and $\beta_0 \coloneqq \beta(0)$, we have}
\begin{equation} \label{rewriting a}
	\begin{aligned}
		a(\uh, q(\thetah) \phih)
		=&\,
		a(\uh, \projRitz [q(\thetah) \phih])
		=
		-(\Delta_h \uh, \projRitz [q(\thetah) \phih])_{L^2}, \\
		a(\uht, \tbeta(\thetah) \phih)	=&\,
		a(\uht, \projRitz [\tbeta(\thetah) \phih])
		=
		-(\Delta_h \uht, \projRitz [\tbeta(\thetah) \phih])_{L^2}.
	\end{aligned}
\end{equation}
The semi-discrete heat equation is given by
\begin{subequations} \label{semi-discret problem heat}
	\begin{equation+} \label{discrete heat}
		(\thetaht, \phih)_{L^2}
		+
		\kappa a(\thetah, \phih) +\nu(\thetah, \phih)_{L^2} =\, (\calQ(\uh, \uth, \thetah), \phih)_{L^2}
	\end{equation+}
	for all $\phih \in \Vh$, $t \in [0,T]$, with
	\begin{equation+} \label{discrete heatIC}
		{\thetah}_{\vert t=0}= \thetazeroh.
	\end{equation+}
\end{subequations}
Our main theoretical result establishes \emph{a priori} error bounds for $(\uh, \thetah)$ in the energy norm.
\begin{theorem}[\emph{A priori} error estimate]\label{thm main}
	Let the assumptions made on the temperature-dependent functions in Section~\ref{Sec: Assumptions functions} and on the exact solution $(u, \theta)$ of \eqref{coupled_ibvp} in Section~\ref{Sec: Assumptions exact solution} hold with $\poldeg \geq 1$. Assume that $f$, $f_h \in L^2(0,T; \Ltwo)$ are such that
	\begin{equation} \label{assumption_accuracy_fh}
		\begin{aligned}
			\|\fu-\fuh\|_{\LtwoLtwo}
			\leq C h^\poldeg
		\end{aligned}
	\end{equation}
	and that the approximate initial data are chosen as the Ritz projections of the exact ones; that is,
	\begin{equation} \label{eq:init_choice_space}
		(\uh(0), \partial_t \uh(0))= (\projRitz u_0, \projRitz u_1), \quad 	\thetah(0)= \projRitz \theta_0.
	\end{equation}
	Then, there exist $h_0>0$ and $r>0$,  independent of $h$, such that for all $h\leq h_0$ and
	\begin{equation} \label{smallness condition u theta r}
		\begin{aligned}
			&\begin{multlined}[t]
				\norm{{\beta(\theta)-\beta_0}}_{\LinfLinf}
				+ \|\kW(\theta) u\|_{\CLinf}
				+
				\|\kK (\theta) \ut\|_{\CLinf}
%				\\
				+
				\|\alpha(\theta) \ut\|_{\CLinf}
				\leq r,
			\end{multlined}
		\end{aligned}
	\end{equation}
{with $\beta_0 \coloneqq \beta(0)$},	problem \eqref{semi-discret problem wave}, \eqref{semi-discret problem heat} has a unique solution $(\uh, \thetah) \in H^2(0,T; \Vh) \times H^1(0,T; \Vh)$, which satisfies the following error bound:
	\begin{equation} \label{FE_final_est}
		\begin{aligned}
			&\begin{multlined}[t]
				\|\delt^2(u-\uh)\|_{\LtwotLtwo}
				+
				\|\nabla \delt (u-\uh)\|_{\LinftLtwo}
				+
				\|\nabla (u-\uh)\|_{\LinftLsix}
				\\ \hspace*{1.4cm}
				+
				\| \delt (\theta-\thetah)\|_{\LinftLtwo} +	\| \nabla (\theta-\thetah)\|_{\LinftLsix}
				\leq	C(\norm{u}_{\Xu}, \norm{\theta}_{\Xtheta})  h^{\poldeg}.
			\end{multlined}
		\end{aligned}
	\end{equation}
\end{theorem}
~\\
\paragraph{\bf Discussion of the main result} Theorem~\ref{thm main} establishes sufficient conditions for the optimal order of convergence of $(\uh, \thetah)$ in the energy norm. Let us discuss some of the assumptions. The need for using the Ritz projection of the initial values comes from bounds involving $\Deltah$ applied to the initial error; see Sections~\ref{Sec: estimate wave} and~\ref{Sec: estimates semidiscrete heat} for details. A different choice, say, for example, the nodal interpolation, would lead to 
{a theoretical}
 order reduction in the error analysis.
{
Such artificial restrictions on approximate data are present in the literature in the numerical analysis of other nonlinear models; see, e.g., the seminal work by Makridakis in~\cite[Theorem~2.1]{Mak93}.}\\
\indent The assumption that the exact temperature should satisfy \sloppy $\theta \in L^2(0,T;  W^{\poldeg+1, \dimension+\delta }(\Omega))$ comes from the need to estimate the error in the temperature dependent coefficients $\beta$ and $q$. In particular, we need to employ the following estimate
for  $w \in \{q, \tbeta\}$ (see \eqref{est h q beta} below):
\begin{equation}
	\begin{aligned}
		\| w(\theta) - w (\thetah)\|_{\LtwotWonedelta} \lesssim&\,  \|\theta -\thetah\|_{\LtwotWonedelta} \\
		\lesssim&\,   \|\thetah -\Rhtheta \|_{\LtwotWonedelta}+\|\theta -\Rhtheta\|_{\LtwotWonedelta},
	\end{aligned}
\end{equation}
and then further use the bound $\|\theta -\Rhtheta\|_{L^2(W^{1, \dimension+\delta }(\Omega))} \lesssim h^{\poldeg} \|\theta\|_{L^2(W^{\eta, \dimension+\delta }(\Omega))}$ (see Section~\ref{Sec: Ritz proj} for the approximation properties of the Ritz projection).\\
\indent Similarly, the assumption that $\theta \in L^2(0,T; W^{\eta+1,\infty}(\Omega))$
comes from needing to estimate the error in the temperature-dependent coefficients $\kW$ and $\kK$. In particular, we will employ the following bound  (see \eqref{est diff k terms} below):
\begin{equation}
	\begin{aligned}
		&\|\kW(\theta)-\kW(\thetah))\|_{\LtwotLinf} + \|\kK(\theta)-\kK(\thetah))\|_{\LtwotLinf}\\
		\lesssim&\, \|\theta-\thetah\|_{\LtwotLinf}\\
		\lesssim&\, \|\thetah-\Rhtheta\|_{\LtwotLinf}+\|\theta-\Rhtheta\|_{\LtwotLinf}
	\end{aligned}
\end{equation}
and then further need to rely on the fact that $\|\theta-\Rhtheta\|_{\LtwoLinf} \lesssim h^{\poldeg} \|\theta\|_{L^2(W^{\eta+1, \infty}(\Omega))}$.
{Even though this regularity assumption could be improved, we still need it for a technical estimate within the proof of Lemma~\ref{lemma: est tbeta q terms}.}\\
\indent Finally, let us note that there is a large literature available on the discretization of nonlinear wave equations originating from the seminal work \cite{BalD89}. However, they do not consider the coupled wave-heat case, and we thus refrain from a further discussion.

\subsection{Organization of the rest of the paper} The rest of the paper is organized as follows. In Section~\ref{Sec: Approach Aux results}, we provide background results on parabolic estimates which are used in the well-posedness and error analysis of the semi-discrete problem, as well as certain useful properties of the Ritz projection and known embedding and inverse estimates. In Section~\ref{Sec: Existence}, we prove that the semi-discrete problem has an accurate solution, however, on a possibly $h$-dependent time interval. Toward prolonging the existence of this solution to $[0,T]$, we then focus on deriving uniform energy estimates for the wave and heat subproblems in Sections~\ref{Sec: estimate wave} and~\ref{Sec: estimates semidiscrete heat}, respectively. These are combined in Section~\ref{Sec: proof of the main result} to {prove the} main theoretical result of this work stated in Theorem~\ref{thm main}. Finally, in Section~\ref{Sec: Numerical experiments} we validate the theoretical convergence rate through numerical experiments and provide additional numerical examples, where we show the performance of the model and developed numerical schemes.

\section{The approach and auxiliary results} \label{Sec: Approach Aux results}

Our numerical analysis follows by first proving the existence of a solution $(\uh, \thetah)$ on a possibly discretization-dependent time interval $[0, \finalth]$ and then extending the existence to $[0,T]$ by means of a suitable uniform bound on this solution.  This approach is in the general spirit of, e.g.,~\cite{hochbruck2022error, Doe24, dorich2024robust}, which have investigated single-physics wave models. The focal and most delicate point of the numerical analysis here is the derivation of a suitable energy bound for the nonlinear wave-heat system.  To this end, since the acoustic component is strongly damped, the idea is to see the wave-heat system %semi-discrete problem
 in the following parabolic form:
 {
 \begin{equation} \label{coupled_problem parabolic exact}
 	\left\{ \begin{aligned}
 		&\begin{aligned}
 			\utt- \beta_0 \Delta \ut & = q(\theta)
 			\Deltah \uh+\tbeta(\theta)\Delta_h \uht
 			%			\\ &\quad
 			- \calN(u, \ut, \uhtt, \nabla u, \nabla \ut, \theta)+f,
 		\end{aligned}\\[1mm]
 		& 	\thetat -\kappa\Delta \theta+ \nu\theta = \calQ(u, \ut, \theta),
 	\end{aligned} \right.
 \end{equation}
where $\tbeta(\theta) \coloneqq \beta(\theta)-\beta_0$.  We then discretize it as follows: }
%\begin{equation} \label{coupled_problem parabolic}
%	\left\{ \begin{aligned}
%		&\begin{aligned}
%			\uhtt- \beta_0 \Delta_h \uht & = q(\thetah)
%			\Deltah \uh+\tbeta(\thetah)\Delta_h \uht
%%			\\ &\quad
%			- \calN(\uh, \uht, \uhtt, \nabla \uh, \nabla \uht, \thetah)+f_h,
%		\end{aligned}\\[1mm]
%		& 	\thetaht -\kappa\Delta_h \thetah+ \nu\thetah = \calQ(\uh, \uht, \thetah).
%	\end{aligned} \right.
%\end{equation}
%where $\tbeta(\thetah) = \beta(\thetah)-\beta_0$. 
%
	\begin{equation} \label{coupled_problem parabolic}
		\left\{ 
		\begin{aligned}
			&\begin{aligned}
				\begin{multlined}[t]
					(\uhtt, \phih)_{L^2}
					 +
					\beta_0  a(\pht,  \phih)
					 = 
					 -
					  a(\uh, q(\thetah) \phih)
					-
					a(\pht,\tbeta(\thetah) \phih) 
					\\\hspace*{44mm}
					-(\calN(\uh, \uht, \uhtt, \nabla \uh, \nabla \uht, \thetah), \phih)_{L^2}
					+ (\fph, \phih)_{L^2} ,
				\end{multlined}
			\end{aligned}
			\\[1mm]
			&(\thetaht, \psi_h)_{L^2}
			+
			\kappa a(\thetah, \psi_h) +\nu(\thetah, \psi_h)_{L^2} =\, (\calQ(\uh, \uth, \thetah), \psi_h)_{L^2}
		\end{aligned} 
\right.
	\end{equation}
%	%
for all $\phih, \psi_h \in  \Vh$.

This setup will allow us to exploit, to a certain extent, estimates for semi-discrete parabolic problems. For this reason, we present next two estimates for linear parabolic problems that will be used in Sections~\ref{Sec: estimate wave} and~\ref{Sec: estimates semidiscrete heat}.

\subsection{A maximal regularity estimate for linear parabolic problems}

Given $\gh \in L^2(0,T; \Vh)$, $b_0>0$ and $\nu \geq 0$, consider the problem
\begin{equation} \label{semidiscrete parabolic}
	\ip{\delt \wh }{\varphi_h} + b_0 a(\wh, \varphi_h)+ \nu \ip{\wh}{\varphi_h} = \ip{g_h }{\varphi_h}, \qquad \forall \varphi_h \in \Vh.
\end{equation}
The finite element analysis of this problem with homogeneous initial data using a maximal $L^p$ regularity approach can be found, for example, in~\cite[Theorem 1.1]{LiS17}. For completeness, we present here the derivation of the $L^2$-based energy bound, where compared to~\cite{LiS17} we allow for non-zero initial data.

\begin{lemma} \label{lemma: L2L2 estimate}
	Let $\gh \in L^2(0,T; \Vh)$.	The solution of \eqref{semidiscrete parabolic} satisfies
	\begin{equation} \label{eq:energy_estimate}
		\begin{aligned}
			& \begin{multlined}[t]
				\frac12 \int_0^t \norm{\delt \wh(s)}_{\Ltwo}^2 \ds
				+
				\Bigl(\frac{b_0}{2} +1\Bigr) \norm{\nabla \wh}_{\Ltwo}^2 \big|_0^t
				+\frac{\nu}{2} \norm{ \wh}_{\Ltwo}^2 \big|_0^t \\
				+
				\frac{b_0}{2} \int_0^t  \norm{\Delta_h \wh(s)}_{\Ltwo}^2 \ds	+
				\nu \int_0^t  \norm{\nabla \wh(s)}_{\Ltwo}^2 \ds
			\end{multlined}
			\\
			&\leq
			\frac12
			\Bigl(1+  \frac{1}{b_0} \Bigr)
			\int_0^t \norm{\gh (s)}_{\Ltwo}^2 \ds.
		\end{aligned}
	\end{equation}
\end{lemma}
\begin{proof}
	We test \eqref{semidiscrete parabolic} with $\varphi_h = \delt \wh$ and integrate over $(0,t)$ for $t \in (0,T)$ to obtain
	\begin{equation}
		\int_0^t \norm{\delt \wh(s)}_{\Ltwo}^2 \ds
		+
		\frac{ b_0}{2} \norm{\nabla \wh}_{\Ltwo}^2 \big|_0^t +
		\frac{ \nu}{2} \norm{ \wh}_{\Ltwo}^2 \big|_0^t
		=
		\int_0^t\ip{g_h (s) }{\delt \wh (s) } \ds.
	\end{equation}
	By choosing instead  $\varphi_h = - \Delta_h \wh$, we obtain
	\begin{equation}
		\begin{aligned}
			& \begin{multlined}
				\norm{\nabla \wh}_{\Ltwo}^2 \big|_0^t
				+
				b_0 \int_0^t  \norm{\Delta_h \wh(s)}_{\Ltwo}^2 \ds	+
				\nu \int_0^t  \norm{\nabla \wh(s)}_{\Ltwo}^2 \ds
%				\\
				=
				\int_0^t\ip{\gh (s) }{\Delta_h \wh (s)} \ds.
			\end{multlined}
		\end{aligned}
	\end{equation}
	Adding the estimates and employing Young's inequality yields \eqref{eq:energy_estimate}.
\end{proof}
The bound in Lemma~\ref{lemma: better estimate parabolic} will be employed in the error analysis of the semi-discrete wave subproblem in Section~\ref{Sec: estimate wave} with $w_h =\delt (\projRitz u -\uh)$.

\subsection{Additional estimate for a more regular right-hand side} When working with the semi-discrete heat equation, we will have a relatively regular in time right-hand side due to the properties of the semi-discrete pressure field. For this reason, we also derive here an additional bound for parabolic problems that assumes more regularity in time of the right-hand side.
\begin{lemma} \label{lemma: better estimate parabolic}
	Let $\gh \in H^{1}(0,T; V_h)$.	Then the solution of \eqref{semidiscrete parabolic} satisfies
	\begin{equation} \label{est heat}
		\begin{aligned}
			&	\|\delt \wh(t)\|^2_{\Ltwo}
			+\int_0^t \norm{\nabla \delt \wh(s)}_{\Ltwo}^2 \ds
			+
			\frac{ b_0}{2} \norm{\Delta_h \wh(t)}_{\Ltwo}^2  +
			\frac{ \nu}{2} \norm{\nabla \wh(t)}_{\Ltwo}^2 \\
			\lesssim&\, e^{CT} (\|g_h\|^2_{H^{1}(\Ltwo)}+\|\Delta_h \wh(0)\|^2_{\Ltwo}+\|\nabla \wh(0)\|^2_{\Ltwo}).
		\end{aligned}
	\end{equation}
\end{lemma}
\begin{proof}
	Note that $\gh \in H^{1}(0,T; V_h) \hookrightarrow C([0,T]; V_h)$. Testing with $\varphi_h = - \Delta_h \delt \wh$ and integrating in time leads to
	\begin{equation}
		\begin{multlined}
			\int_0^t \norm{\nabla \delt \wh(s)}_{\Ltwo}^2 \ds
			+
			\frac{ b_0}{2} \norm{\Delta_h \wh}_{\Ltwo}^2 \big|_0^t +
			\frac{ \nu}{2} \norm{\nabla \wh}_{\Ltwo}^2 \big|_0^t
%			 \\
			=
			-\int_0^t\ip{\gh (s) }{\Delta_h \delt \wh (s) } \ds.
		\end{multlined}
	\end{equation}
	To treat the right-hand side, we integrate by parts in time:
	\begin{equation}
		\begin{aligned}
			-\int_0^t\ip{g_h (s) }{\Delta_h \delt \wh (s) } \ds  & = - \ip{g_h (s) }{\Delta_h  \wh (s) } \big|_0^t
%			\\ &\qquad
			+ \int_0^t\ip{\delt \gh (s) }{\Delta_h \wh (s) } \ds.
		\end{aligned}
	\end{equation}
	Using Young's inequality, we have for any $\eps>0$
	\begin{equation}
		\begin{aligned}
			&	\int_0^t \norm{\nabla \delt \wh(s)}_{\Ltwo}^2 \ds
			+
			\frac{ b_0}{2} \norm{\Delta_h \wh}_{\Ltwo}^2 \big|_0^t +
			\frac{ \nu}{2} \norm{\nabla \wh}_{\Ltwo}^2 \big|_0^t \\
			\lesssim&\,\begin{multlined}[t]  \|g_h(t)\|^2_{\Ltwo}+\eps \|\Delta \wh(t)\|^2_{\Ltwo}
				+\|g_h(0)\|^2_{\Ltwo}+ \|\Delta_h \wh(0)\|^2_{\Ltwo}
				\\
				+\int_0^t \|\delt \gh(s)\|^2_{\Ltwo}\ds+\intt \|\Delta_h \wh(s)\|^2_{\Ltwo}\ds.
			\end{multlined}
		\end{aligned}
	\end{equation}
	Then  by choosing $\eps$ sufficiently small and employing \Gronwall's inequality, we obtain
	\begin{equation}
		\begin{aligned}
			&\int_0^t \norm{\nabla \delt \wh (s)}_{\Ltwo}^2 \ds
			+
			\frac{ b_0}{2} \norm{\Delta_h \wh(t)}_{\Ltwo}^2  +
			\frac{ \nu}{2} \norm{\nabla \wh(t)}_{\Ltwo}^2 \\
			\lesssim&\, e^{CT} (\|g_h\|^2_{H^{1}(\Ltwo)}+\|\Delta_h \wh(0)\|^2_{\Ltwo}+\|\nabla \wh(0)\|^2_{\Ltwo}).
		\end{aligned}
	\end{equation}
	Additionally, since $\delt \wh = b_0 \Delta_h \wh - \nu \wh + \gh$, we can bootstrap the above estimate to obtain
	\begin{equation}
		\begin{aligned}
			\|\delt \wh(t)\|^2_{\Ltwo} = &\,		\|b_0 \Delta_h \wh(t) -\nu \wh(t) + \gh(t) \|_{\Ltwo}^2 \\
			\lesssim &\, e^{CT} (\|g_h\|^2_{H^{1}(\Ltwo)}+\|\Delta_h \wh(0)\|^2_{\Ltwo}+\|\nabla \wh(0)\|^2_{\Ltwo}).
		\end{aligned}
	\end{equation}
	By adding the two bounds, we arrive at \eqref{est heat}.
\end{proof}
The bound in Lemma~\ref{lemma: better estimate parabolic} will be employed in the error analysis of the semi-discrete heat subproblem in Section~\ref{Sec: estimates semidiscrete heat} with $w_h = \projRitz \theta -\thetah$.

\subsection{Embeddings and inverse estimates} When deriving estimates in Sections~\ref{Sec: estimate wave} and~\ref{Sec: estimates semidiscrete heat}, we will utilize (discrete) embedding results and inverse estimates for finite element functions in the upcoming error analysis. In particular, the following embedding holds:
\begin{equation}\label{embedding WNdelta}
	H^2(\Omega) \hookrightarrow  \Wonedelta \hookrightarrow L^\infty(\Omega), \quad \dimension \in \{1,2,3\},\ \delta>0.
\end{equation}
For $\phih \in \Vh$, we have
the discrete Sobolev embedding
\begin{equation} \label{discrete_Sobolev_embedding}
	\norm{\phih}_{ \Linf} + \norm{\phih}_{W^{1,6}(\Om)} \leq C \norm{\Delta_h \phih}_{\Ltwo},
\end{equation}
where $C$ is independent of $h$; see, for example,~\cite{Doe24, FujSS01,SuzF86}.
%
%\indent
Furthermore, the following inverse estimates are used in the analysis (see, for example, \cite{brenner2008mathematical}) :
\begin{subequations} \label{eq:inv_estimate}
	\begin{align+}
		\norm{ \nabla \varphi_h}_{\Ltwo} &\leq C h^{-1} \norm{ \varphi_h}_{\Ltwo},  \label{eq:inv_estimate1}
		\\
		\norm{\Delta_h \varphi_h}_{\Ltwo} &\leq C h^{-1} \norm{ \nabla \varphi_h}_{\Ltwo}, \label{eq:inv_estimate2}
		\\
		\norm{\varphi_h}_{\Linf} &\leq C h^{-\dimension/p} \norm{ \varphi_h}_{L^p(\Omega)}, \label{eq:inv_estimate3}
	\end{align+}
\end{subequations}
for $\varphi_h \in \Vh$ and $p\in[1,\infty]$,
with constants independent of $h$.
\\
\indent
We also recall the following bounds for the interpolant, 
which can also be found in \cite{brenner2008mathematical}:
\begin{equation}  \label{eq:Ih_approx}
	\norm{\varphi - \Ih \varphi}_{L^p(\Om)}
	+
	h
	\norm{\varphi - \Ih \varphi}_{W^{1,p}(\Om)} \leq C h^{\ell+1}  \norm{\varphi}_{W^{\ell+1,p}(\Om)},
	\qquad \varphi \in W^{\ell+1,p}(\Om),
\end{equation}
for $2 \leq p \leq \infty$ and $1 \leq \ell \leq \eta$.
\subsection{Properties of the Ritz projection} \label{Sec: Ritz proj} In Sections~\ref{Sec: estimate wave} and~\ref{Sec: estimates semidiscrete heat}, we also heavily rely on certain properties of the Ritz projection. We collect these results here for convenience. For the purpose of estimating the approximation errors in the upcoming analysis, we need the following approximation result
for $0\leq \ell \leq \poldeg$:
\begin{align} \label{approx properties projRitz}
	%
	%	\norm{\varphi - \projRitz \varphi}_{L^p(\Om)}
	%	+
	h \norm{\varphi - \projRitz \varphi}_{W^{1,p}(\Om)} &\leq C h^{\ell+1}  \norm{\varphi}_{W^{\ell+1,p}(\Omega)},
	\quad \varphi \in W^{\ell+1,p}(\Omega),
\end{align}
for all $2\leq p \leq \infty$; see, for example, \cite[Thm.\ 8.5.3]{brenner2008mathematical}.  In particular,
we often employ the stability bound
%
%}
\[
\|\projRitz \varphi\|_{\Linf}
\leq C \|\varphi\|_{W^{1,\infty}(\Omega)}.
\]
Note that one could also replace the right-hand side with the $W^{1,p}$-norm for $p >\dimension$, but for better readability, we employ the above bound.
We further need the estimate
\begin{equation} \label{eq:Ritz_Linf_hhalf}
\norm{\projRitz \varphi - \varphi}_{\Linf} \leq C h^{1/2} \norm{\varphi}_{\Htwo},
\end{equation}
which is obtained by inserting the nodal interpolation operator, using the inverse estimates \eqref{eq:inv_estimate}, and the $L^\infty$-estimate in \cite[Theorem~4.4.20]{brenner2008mathematical}.

\begin{lemma} \label{Lemma: est Ritz product_v2}
Let $\delta>0$, $\dimension \in \{1,2,3\}$, and  $ \mu \in C^{\poldeg+1}(\R)$. Then, there is a constant $C>0$,
% independent of $h$,
such that for all $\phih \in \Vh$ 	it holds
\begin{equation} \label{est Ritz product}
	\norm{\projRitz [ \mu(\psi_h) \, \phih ]}_{\Ltwo} \leq
	\norm{\mu(\psi_h)}_{\Linf}
	\norm{ \phih }_{\Ltwo}
	+
	h^{ \tfrac{\delta}{\dimension+\delta }}
	C(\norm{\psi_h}_{\Wonedelta})
	\norm{ \phih }_{\Ltwo},
\end{equation}
where $C$ is independent of $h$, but depends on $\norm{\psi_h}_{\Wonedelta}$.
\end{lemma}

\begin{proof}
The proof can be found in Appendix~\ref{sec:appendix}.
\end{proof}

%%%%%%%%%%%%%%%%%%%%%%%%%%%%%%%%%%%%%%%%%%%%%%%%%%%%%%%%%%%%%%%%%%%%%%%%%%%%%%%%%
\section{Existence and uniqueness on a discretization-dependent time interval} \label{Sec: Existence}
In this section, we show that the problem has a solution on a possibly $h$-dependent time interval $[0, \finalth]$. In subsequent sections, we will carry out the estimates on this time interval with the goal of obtaining a uniform bound on $(\uh, \thetah)$ in a suitable norm that will allow us to prolong the existence to $[0,T]$. As usual, we split the errors as follows:
\begin{equation}
\begin{aligned}
	u-\uh =&\, (u-\projRitz u) +(\projRitz u -\uh) \,, \\
	\theta-\thetah=&\, (\theta-\projRitz \theta) +(\projRitz \theta -\thetah) \,,
\end{aligned}
\end{equation}
and denote the discrete errors by $\ehu\coloneqq\projRitz u -\uh$ and $\ehtheta \coloneqq\projRitz \theta - \thetah$. \\
\indent We aim to prove that the semi-discrete problem has a solution on the time interval $[0, \finalth]$, where we define $\finalth$ as follows:
\begin{equation} \label{def:finaltimeh}
\begin{aligned}
	\finalth \coloneqq \sup  \Big \{  t \in (0,T] \mid
	\	&\text{a unique solution }	(\uh, \thetah) \in C^{2}([0,t]; \Vh) \times C^1([0,t]; \Vh)
%	 \\ &
	 \text{ of
		\eqref{semi-discret problem wave}
		and
		\eqref{semi-discret problem heat}
		exists and} 
	\\
	&\,  h^{-1/2-\eps}
	(
	\norm{\errhut(s)}_{\Hone}
	+
	\norm{\Delta_h \errhu(s)}_{\Ltwo}
	)
	\leq C_0,  \\
	&\, h^{-1/2-\eps}
	%(
	%\norm{\errhthetat(s)}_{\Ltwo}
	%	+
	\norm{\Delta_h \errhtheta(s)}_{\Ltwo}
	% )
	\leq C_0,
%	\\
%	&
	\,	\text{for all } s \in [0,t]  \Big  \}
\end{aligned}
\end{equation}
for some $\eps \in (0, 1/2)$, and a constant $C_{0}>0$ independent of $h$. The particular choice of terms and norms involved in \eqref{def:finaltimeh} is motivated by the needs of deriving the estimates, as will become apparent below and in Sections~\ref{Sec: estimate wave} and~\ref{Sec: estimates semidiscrete heat}. The first claim of this section concerns the accuracy of the approximate initial data.

\begin{lemma} \label{Lemma: Initial error}
Under the assumptions of Theorem~\ref{thm main}, with the approximate initial values chosen to be  $(\uh(0), \uth(0), \thetah(0))=(\projRitz u_0, \projRitz u_1, \projRitz \theta_0)$,
we have
\begin{equation} \label{init est 1}
	\norm{\errhut(0)}_{\Hone}+\norm{\Delta_h \errhu(0)}_{\Ltwo}  \leq C h^\poldeg
\end{equation}
and
\begin{equation}\label{init est 2}
	%	\|\errhtheta(0)\|_{\Hone}+
	\|\Dh \errhtheta(0)\|_{\Ltwo} \leq C h^\poldeg.
\end{equation}
\end{lemma}
\begin{proof}
With our choice of the initial data we immediately have	$\errhu(0)=\errhut(0)=\errhtheta(0)=0$, and thus the statement trivially holds.
\end{proof}

We next tackle the existence of a unique solution of the semi-discrete system on a possibly $h$-dependent time interval, which will allow us to conclude that $\finalth>0$.

\begin{proposition} \label{Prop: finalth positive}
Under the assumptions of Theorem~\ref{thm main}, we have $\finalth>0$.
\end{proposition}
\begin{proof}
The statement will follow by considering a first-order rewriting of the system  and applying on it a local version of the Picard--Lindel\"of theorem on the open set
\begin{equation}\label{def_Uh}
	\begin{aligned}
		U_{h} \coloneqq \Bigl \{ (\uh,  \uht, \thetah) \in  V_h^3:& \
		\norm{\kW(\thetah)\uh}_{\Linf}
		+
		\norm{\kK(\thetah)\uht }_{\Linf}
		< r +\delta \Bigr \}
	\end{aligned}
\end{equation}
with $\delta>0$ to be determined below,
and radius $r$ from Theorem~\ref{thm main}.
To see that the initial values belong to $U_h$, we can uniformly bound $\thetah(0)$ using the inverse estimate in \eqref{eq:inv_estimate3} and discrete embedding \eqref{discrete_Sobolev_embedding} as follows:
\begin{equation}	\label{eq:thetat_theta_Linf}
	\begin{aligned}
		\|\thetah (0)  - \theta(0) \|_{\Linf} 
		 \lesssim&\,h^{-\dimension/6} 
%		 \|\thetah(0)-\projRitz \theta(0)\|_{\Lsix} 
 		 \| \ehtheta(0)\|_{\Lsix} 
		 +\|\projRitz \theta(0) - \theta(0)\|_{\Linf}
		 \\
		\lesssim&\,h^{-\dimension/6}
%		 \|\nabla(\thetah (0) -\projRitz \theta (0) )\|_{\Ltwo} 
		  \| \nabla \ehtheta(0)\|_{\Ltwo} 
		 +\|\projRitz \theta (0)  - \theta (0) \|_{\Linf}
		\\
		\lesssim&\,  h^{-\dimension/6} \|\Delta_h \ehtheta (0)\|_{\Ltwo}+h^{1/2} \|\theta\|_{\LinftHtwo} \leq  C h^{\eps} ,
	\end{aligned}
\end{equation}
where we have used Lemma~\ref{Lemma: Initial error} in the last step.
Considering the $\uh(0)$ and $\uht(0)$ terms, we similarly have
\begin{equation} \label{eq:h_uh_Linf}
	\begin{aligned}
		\|\uh (0) - u (0) \|_{\Linf}
		&\lesssim
		h^{-\dimension/6 }\|\Delta_h \ehu (0)\|_{\LinftLtwo}+ h^{1/2} \|u\|_{\LinftHtwo}\leq  C h^{\eps},
		\\
		\|\uht(0) - \ut (0)\|_{\Linf}
		&\lesssim
		h^{-\dimension/6} \|\nabla \pt \ehu(0)\|_{\Ltwo} + h^{1/2} \|\ut\|_{\LinftHtwo} \leq  C h^{\eps},
	\end{aligned}
\end{equation}
using \eqref{eq:Ritz_Linf_hhalf}. Combining the three estimates yields
\begin{equation} \label{eq:k_u_products_Linf}
	\begin{aligned}
		\norm{\kW(\thetah(0))\uh(0) - \kW(\theta(0))u(0)}_{\Linf} &\leq C h^{\eps} ,
		\\
		\norm{\kK(\thetah(0))\uht(0) - \kK(\theta(0))\ut(0) }_{\Linf} &\leq C h^{\eps},
	\end{aligned}
\end{equation}
and thus $(\solh(0), \pt \solh(0), \thetah(0)) \in U_{h}$ for any $\delta$
if $h \leq h_0$ is small enough.

To state the semi-discrete problem in a compact manner, we introduce the operator  $\Lambda_h$ defined by
\begin{equation}
	\ip{\Lambda_h(\uh, \uht, \thetah) \varphi_h}{\psi_h} \coloneqq 	\ip{(1+2\kW(\thetah)\uh+2\kK(\thetah)\uht) \varphi_h}{\psi_h}
\end{equation}
for $\varphi_h,\psi_h \in \Vh$.	Further, given $\mu \in C^{0,1}(\R)$, we introduce the operator
\begin{equation}
	\begin{aligned}
		\ip{A_h (\mu(\thetah))\varphi_h}{\psi_h} \coloneqq &\,(\nabla \varphi_h, \nabla(\mu(\thetah) \psi_h))_{L^2}.
	\end{aligned}
\end{equation}
The 
%time-differentiated 
semi-discrete problem can then be written as (with $L^2$-projection $\pi_h$)
\begin{equation} \label{rewritten_problem}
	\begin{aligned}
		\Lambda_h(\uh, \uht, \thetah)\pt^2 \solh
		=&\, \begin{multlined}[t]
			A_h(q(\thetah))\solh
			+
			A_h(\beta(\thetah)) \pt \solh
			\\\hspace*{1.8cm}	- 2 \pi_h \bigl(
			\kW(\thetah)(\uht)^2 + \ell \nabla \uh \cdot \uht
			\bigr)				-
			f_h,\end{multlined}\\
		\thetaht=&\,-\kappa \Delta_h \thetah-\nu \thetah + \pi_h \calQ(\uh, \uht, \thetah)  .
	\end{aligned}
\end{equation}
Note that the operator $\Lambda_h$ is invertible on $U_h$ for small enough $r$,
since we can find $\gamma,\delta>0$, independent of $h$, such that
\begin{equation} \label{nondegeneracy}
	1+2\kW(\thetah)\uh+2\kK(\thetah)\uht \geq \gamma>0.
\end{equation}

\indent Therefore, the semi-discrete problem can be further rewritten as a first-order system for $\boldsymbol{v}_h \coloneqq (\solh, \pt \solh, \thetah)^T$:
\begin{equation} \label{rewritten_problem first order}
	\left\{	\begin{aligned}
		\,	\pt  \boldsymbol{v}_h =&\, F (\boldsymbol{v}_h), \\
		\,	\boldsymbol{v}_h(0)=&\, (\uzeroh, \uoneh, \thetazeroh)^T,
	\end{aligned} \right.
\end{equation}
where the right-hand side is given by
\begin{equation} \label{def F}
	\begin{aligned}
		&F ((\solh, \pt \solh, \thetah)^T) \\
		\coloneqq &\,\begin{multlined}[t] \Bigl( \uht,\ \left(\Lambda_h(\uh, \uht, \thetah)\right)^{-1} \left(A_h(q(\thetah))\solh
			+
			A_h(\beta(\thetah)) \pt \solh
			-
			\pi_h \bigl( 2\kW(\thetah)(\uht)^2 \bigr)\right.
			\\
			\left.- 	\pi_h \bigl( 2 \ell \nabla \uh \cdot \uht	 \bigr)			-
			f_h \vphantom{A_h}\right), \ -\kappa \Delta_h \thetah-
			\nu \thetah + \pi_h \calQ(\uh, \uht, \thetah) \Bigr)^T.
		\end{multlined}
	\end{aligned}
\end{equation}
Furthermore, system  \eqref{rewritten_problem first order} has a locally Lipschitz continuous right-hand side \eqref{def F}. Indeed, Lipschitz continuity of the right-hand side follows by the fact that $V_h$ is a finite-dimensional space in which we can use inverse estimates
\eqref{eq:inv_estimate1}--\eqref{eq:inv_estimate2}.   \\
\indent Thus by the local version of the Picard--Lindel\"of theorem, a \sloppy unique solution $(\solh, \thetah) \in C^2([0,T]; V_h) \times C^1([0,T]; V_h)$ of
\eqref{rewritten_problem}, supplemented with approximate initial data, exists on $[0,\tilde{t}]$ for some $\tilde{t}>0$.
%\\
%%
Since the initial errors in Lemma~\ref{Lemma: Initial error} in fact vanish,
we have by the continuity and the equivalence of norms on $V_h^3$ that the errors $\errhu$ and $\errhtheta$
still satisfy the bounds in \eqref{def:finaltimeh} for a short time.
Therefore, we conclude that $\finalth>0$.
\end{proof}
We also prove two uniform boundedness results on $[0, \finalth]$ that will be useful in the next step of the error analysis.
\begin{lemma} \label{lemma: uniform bounds}
Let the assumptions of Theorem~\ref{thm main} hold. Then the following bounds hold on $[0, \finalth]$:
\begin{equation}
	\|\uh\|_{\LinftLinf} + \|\nabla \uh\|_{\LinftLinf}+ \|\uht\|_{\LinftLinf} \lesssim 1
\end{equation}
and
\begin{equation}
	\|\thetah\|_{\LinftLinf}+\|\thetah\|_{\LinftWdelta}+\|\thetaht\|_{\LinftLp{3}} \lesssim 1.
\end{equation}
\end{lemma}
\begin{proof}
The statement follows by a repeated use of the stability properties of the Ritz projection
and inverse estimates \eqref{eq:inv_estimate1}--\eqref{eq:inv_estimate3}. We have already shown in the previous proof that for $s \in [0, \finalth]$
\begin{equation}
	\begin{aligned}
		\|\uh(s)\|_{\Linf}+	\|\uht(s)\|_{\Linf}
		\lesssim 1.
	\end{aligned}
\end{equation}
Further, by \eqref{discrete_Sobolev_embedding}, we have
\begin{equation}
	\begin{aligned}
		\|\nabla \uh(s)\|_{\Linf} & \lesssim \|\Rhu(s)\|_{\Woneinf} + \|\errhu(s)\|_{\Woneinf}
%		\\
%		&
		 \lesssim
		\|u(s)\|_{\Woneinf} + h^{-\dimension/6}\|\Delta_h \errhu(s)\|_{\Ltwo} .
	\end{aligned}
\end{equation}
Similarly, for $\dimension+\delta  \leq 6$, by \eqref{discrete_Sobolev_embedding}, we have
\begin{equation}
	\begin{aligned}
		\|\thetah(s)\|_{\Linf}
		&\leq
		\|\Rhtheta (s)\|_{\Wonedelta}
		+
		\|\errhtheta(s)\|_{\Wonedelta}
%		\\		&
		 \lesssim
		\|\Rhtheta (s)\|_{\Wonedelta}
		+
		\|\Delta_h \errhtheta(s)\|_{\Ltwo}.
	\end{aligned}
\end{equation}
Lastly, inserting the equation \eqref{coupled_problem parabolic} for $\thetah$, the stability of the $L^2$-projection $\pi_h$ in $L^3$, and the relation
$	\Delta_h \projRitz = \pi_h\Delta$, we obtain
\begin{align}
	\|\thetaht(s)\|_{\Lp{3}}
	&\lesssim
	\|	\Delta_h \thetah\|_{\Lp{3}}
	+
	\|	\thetah 	\|_{\Lp{3}}
	+
	\| \calQ(\uh, \uht, \thetah)\|_{\Lp{3}}
	\\
	&\lesssim
	h^{-\dimension/6} \|	\Delta_h \errhtheta(s)\|_{\Lp{2}}
	+
	\| \pi_h\Delta \theta \|_{\Lp{3}}
	+
	\|	\thetah 	\|_{\Lp{3}}
	\\
	&\qquad	+
	C( \|\uh(s)\|_{\Linf} , \|\uht(s)\|_{\Linf} , \|\thetah(s)\|_{\Linf}   )
	\lesssim 1 \,.
\end{align}
The definition of $\finalth$ in \eqref{def:finaltimeh} closes the proof.
\end{proof}

As a corollary of Lemma~\ref{lemma: uniform bounds}, on account of the assumptions made on the temperature-dependent medium parameters, we also have the following uniform bounds.
\begin{corollary} \label{corollary: boundedness} Under the assumptions of Theorem~\ref{thm main}, we have
\begin{equation}
	\begin{multlined}
		\|\tbeta(\thetah)\|_{\LinftLinf}+\|q(\thetah)\|_{\LinftLinf}
%		\\
		+	\|\kW(\thetah)\|_{\LinftLinf}+\|\kK(\thetah)\|_{\LinftLinf}\\+\|\alpha(\thetah)\|_{\LinftLinf} \lesssim 1
	\end{multlined}
\end{equation}
on $[0, \finalth]$.
\end{corollary}
In the next step of the \emph{a priori} error analysis, we wish to derive a uniform estimate for $(\uh, \thetah)$ that will allow us to show that \[(\solh(\finalth), \pt \solh(\finalth), \thetah(\finalth)) \in U_h,\]
%In other words, that
%\begin{equation}
%	\begin{aligned}
%			&\,  h^{-1/2-\eps}
%		(
%		\norm{\errhut(\finalth)}_{\Ltwo}
%		+
%		\norm{\Delta_h \errhu(\finalth)}_{\Ltwo}
%		+
%		\|\nabla \errhut(\finalth)\|_{\Ltwo}
%		)
%		< C_0,  \\
%		&\, h^{-1/2}
%		(
%		\norm{\errhthetat(\finalth)}_{\Ltwo}
%		+
%		\norm{\Delta_h \errhtheta(\finalth)}_{\Ltwo}
%		)
%		< C_0 ,
%	\end{aligned}
%\end{equation}
which will lead us to the conclusion that $(\uh, \thetah)$ exists on $[0,T].$ We focus first on the acoustic subproblem and estimating $\errhu = \Rhu-\uh$.

\section{Estimates for the semi-discrete wave subproblem} \label{Sec: estimate wave}
In this section, our aim is to derive an energy estimate for $\errhu= \projRitz u-\uh$ on $[0, \finalth]$ with $\finalth$ defined in \eqref{def:finaltimeh}. Toward estimating $\errhu$, we observe first that the Ritz projection of $u$ satisfies
\begin{equation}
\begin{aligned}
	&(\Rhptt, \phih)_{L^2}+\beta_0  a(\Rhpt, \phih)\\
	=&\,\begin{multlined}[t]
		( \Delta u,  q(\theta) \phih)_{L^2}+	( \Delta \pt u,  \tbeta(\theta) \phih)_{L^2}
%		\\
		+(\calN(\Rhu, \Rhut, \Rhutt, \nabla \Rhu, \nabla \Rhut, \theta), \phih)+(\fp, \phih)_{L^2}+	(\wtdeltap, \phih)_{L^2},
	\end{multlined}
\end{aligned}
\end{equation}
where the defect is given by
\begin{equation} \label{defect u}
\begin{aligned}
	&(\wtdeltap, \phih)_{L^2}\\ \coloneqq &\, \begin{multlined}[t] (\Rhptt-\ptt, \phih)_{L^2}
%		\\
		\!-(\calN(u, \ut, \utt, \nabla u, \nabla \ut, \theta)-\calN(\Rhu, \Rhut, \Rhutt, \nabla \Rhu, \nabla \Rhut, \theta), \phih)_{L^2}.
	\end{multlined}
\end{aligned}
\end{equation}
We can estimate the defect using the following result.
\begin{lemma}\label{lemma: est defect u}
Under the assumptions of Theorem~\ref{thm main}, the following estimate holds:
\begin{equation}
	\begin{aligned}
		\|\wtdeltap\|_{\LtwoLtwo} \leq C (\norm{u}_{\Xu}) h^{\poldeg}.
	\end{aligned}
\end{equation}
\end{lemma}
\begin{proof}
The proof follows by rewriting the difference of the $\calN$ terms as follows:
\begin{equation}
	\begin{aligned}
		& \calN(u, \ut, \utt, \nabla u, \nabla \ut, \theta)-\calN(\Rhu, \Rhut, \Rhutt, \nabla \Rhu, \nabla \Rhut, \theta) \\
		& = 2\kW(\theta)((u-\Rhu) \utt + \Rhu(\utt-\Rhutt)) + 2\kW(\theta)(\ut-\Rhut)(\ut+\Rhut) \\ &\quad
		+2\kK(\theta)((\ut-\Rhut) \utt + \Rhut (\utt-\Rhutt)) \\ & \quad
		+2 \ell \nabla (u-\Rhu) \cdot \nabla \ut +2 \ell \nabla \Rhu \cdot \nabla (\ut-\Rhut)
	\end{aligned}
\end{equation}
and using the properties of the Ritz projection. Indeed, we have
\begin{equation}
	\begin{aligned}
		&\|2\kW(\theta)((u-\Rhu) \utt + \Rhu(\utt-\Rhutt)) 
%		\\
%		& \qquad\qquad\qquad
		 + 2\kW(\theta)(\ut-\Rhut)(\ut+\Rhut)\|_{\LtwoLtwo} \\
		\lesssim &\,%\,\begin{multlined}[t]
			\|\kW(\theta)\|_{\LinfLinf}\|u-\Rhu\|_{\LinfLtwo}\|\utt\|_{\LtwoLinf}\\&\quad
			+\|\Rhu\|_{\LinfLinf}\|\utt-\Rhutt\|_{\LtwoLtwo}\\&\quad
			+\|\kW(\theta)\|_{\LinfLinf} \|\ut-\Rhut\|_{L^2(L^4(\Omega))}\|\ut+\Rhut\|_{L^2(L^4(\Omega))}.
			%\end{multlined}
		\end{aligned}
	\end{equation}
	Next,
	\begin{equation}
		\begin{aligned}
			&\|2\kK(\theta)((\ut-\Rhut) \utt + \Rhut (\utt-\Rhutt)) \|_{\LtwoLtwo} \\
			\lesssim &\, \|\kK(\theta)\|_{\LinfLinf} \|\ut-\Rhut\|_{\LinfLtwo}\|\utt\|_{\LtwoLinf}
%			\\
%			& \quad
			 +\|\Rhut\|_{\LinfLinf}\|\utt-\Rhutt\|_{\LtwoLtwo}.
		\end{aligned}
	\end{equation}
	Finally,
	\begin{equation}
		\begin{aligned}
			&\|\ell \nabla (u-\Rhu) \cdot \nabla \ut +\ell \nabla \Rhu \cdot \nabla (\ut-\Rhut) \|_{\LtwoLtwo} \\
			\lesssim&\, \|\nabla (u-\Rhu)\|_{\LtwoLtwo}\|\nabla \ut\|_{\LinfLinf}
%			\\
%			&\quad
			 + \|\nabla \Rhu\|_{\LinfLinf}\|\ut-\Rhut\|_{\LtwoLtwo}.
		\end{aligned}
	\end{equation}
	Combining the bounds and relying on \eqref{approx properties projRitz} leads to the claim.
\end{proof}
\noindent Hence, the error $\ehp = \Rhp  - \ph$ satisfies the following parabolic problem:
\begin{equation} \label{eq ehu}
	\begin{aligned}
		(\ehutt, \phih)_{L^2}
		+
		\beta_0 a(\ehtu, \phih)
		=&\, (\calFhu, \phih)_{L^2},
	\end{aligned}
\end{equation}
with the right-hand side given by
\begin{equation} \label{def calFhu}
	\begin{aligned}
%		&
			(\calFhu, \phih)_{L^2}
%			\\
		\coloneqq &\, %\begin{multlined}[t]
		(\wtdeltap, \phih)_{L^2}+(\fu-\fuh, \phih)_{L^2}
		+
		(\Delta u, q(\theta) \phih)_{L^2}
		-
		(\Delta_h \uh, \projRitz [q(\thetah) \phih])_{L^2}
		\\
		&  +
		(\Delta \ut, \tbeta(\theta) \phih)_{L^2}
		-
		(\Delta_h \uht, \projRitz [\tbeta (\thetah) \phih])_{L^2}\\
		&\qquad + \big(\calN(\uh, \uht, \uhtt, \nabla \uh, \nabla \uht, \thetah),\phih\big)_{L^2}\\
		& \qquad\qquad-\big(\calN(\Rhu, \Rhut, \Rhutt, \nabla \Rhu, \nabla \Rhut, \theta), \phih\big)_{L^2}.
		%\end{multlined}
	\end{aligned}
\end{equation}
We wish to apply the maximal regularity estimate result of Lemma~\ref{lemma: L2L2 estimate} to this problem. Toward estimating the right-hand side, we can use Lemma~\ref{lemma: est defect u} to bound $\wtdeltap$. The next result will allow us to estimate the difference of $q$ and $\tbeta$ terms in \eqref{def calFhu}.

\begin{lemma} \label{lemma: est tbeta q terms} Under the assumptions of Theorem~\ref{thm main}, the following bounds hold on $[0, \finalth]$:
	\begin{equation} \label{est ut Rh}
		\begin{aligned}
			&\sup_{\norm{\phih}_{\Ltwo} = 1} |	(\Delta \ut, \tbeta(\theta) \phih)_{L^2}
			-
			(\Delta_h \pht, \projRitz [\tbeta (\thetah) \phih])_{L^2} |
			\\
			&\lesssim
			h^\poldeg
			+
			\norm{ \Delta_h \ehut }_{\Ltwo}
			\bigl(	\norm{ \tbeta(\thetah)}_{\Linf} + o(1) \bigr)
			+
			\norm{ \tbeta(\theta) -  \tbeta(\thetah)}_{\Wonedelta}  \,,
		\end{aligned}
	\end{equation}
	and
	\begin{equation} \label{est u Rh}
		\begin{aligned}
			&\sup_{\norm{\phih}_{\Ltwo} = 1} |	(\Delta u, q(\theta) \phih)_{L^2}
			-
			(\Delta_h \pht, \projRitz [q (\thetah) \phih])_{L^2} |
			\\
			&\lesssim
			h^\poldeg
			+
			\norm{ \Delta_h \ehu }_{\Ltwo}
			\bigl(	\norm{ q(\thetah)}_{\Linf} + o(1) \bigr)
			+
			\norm{ q(\theta) -  q(\thetah)}_{\Wonedelta} ,
		\end{aligned}
	\end{equation}
	where the hidden constants are independent of $h$ and $\finalth$.
\end{lemma}
\begin{proof}
	We only prove \eqref{est ut Rh} as estimate \eqref{est u Rh} follows analogously.	We first have the following rewriting:
	\begin{align}
		&
		(\Delta \ut, \tbeta(\theta) \phih)_{L^2}
		-
		(\Delta_h \pht, \projRitz [\tbeta (\thetah) \phih])_{L^2} \\
		=&\,
		a(\pht, \projRitz [\tbeta (\thetah) \phih])
		-
		a( \ut, \tbeta(\theta) \phih)
		\\
		=&\,\begin{multlined}[t]
			a(\pht - \projRitz \ut , \projRitz [\tbeta (\thetah) \phih])
			+
			a(\projRitz \ut , \projRitz [\tbeta (\thetah) \phih])
			-
			a( \projRitz \ut, \tbeta(\theta) \phih)
%			\\
				+
			a( \projRitz \ut -  \ut  , \tbeta(\theta) \phih)
		\end{multlined}
		\\
		=&\,\begin{multlined}[t]
			a(\ehut , \projRitz [\tbeta (\thetah) \phih])
			+
			a(\projRitz \ut , \projRitz [\tbeta (\thetah) \phih - \tbeta(\theta) \phih])
%			\\
			+
			a( \projRitz \ut -  \ut  , (\Id - \projRitz) \tbeta(\theta) \phih)
		\end{multlined}
		\\
		=&\,\begin{multlined}[t]
			- (\Delta_h \ehut , \projRitz [\tbeta (\thetah) \phih])_{L^2}
			-
			(\Delta_h\projRitz \ut , \projRitz [\tbeta (\thetah) \phih - \tbeta(\theta) \phih])_{L^2}
%			\\
				+
			a( \projRitz \ut -  \ut  , (\Id - \projRitz) [\tbeta(\theta) \phih]),
		\end{multlined}
	\end{align}
	where we used the orthogonality of $\projRitz$ in the $a$-inner product.
	Thus, employing Lemma~\ref{Lemma: est Ritz product_v2} yields
	\begin{align}
		&\sup_{\norm{\phih}_{\Ltwo} = 1} |	(\Delta \ut, \tbeta(\theta) \phih)_{L^2}
		-
		(\Delta_h \pht, \projRitz [\tbeta (\thetah) \phih])_{L^2} |
		\\
		&\leq
		\,\begin{multlined}[t]
			\norm{ \Delta_h \ehut }_{\Ltwo}
			\bigl(	\norm{ \tbeta(\thetah)}_{\Linf} + o(1) \bigr)
			+
			\norm{\Delta \pt u}_{\Ltwo}\norm{ \tbeta(\theta) -  \tbeta(\thetah)}_{\Wonedelta}
			\\
			+
			h^\poldeg 		\norm{\ut }_{H^{\poldeg+1}(\Omega)}
			\sup_{\norm{\phih}_{\Ltwo} = 1} \norm{(\Id - \projRitz) [\tbeta(\theta) \phih] }_{\Hone}.
		\end{multlined}
	\end{align}
	By the best approximation property,
	\begin{align}
		\norm{(\Id - \projRitz) [\tbeta(\theta) \phih] }_{\Hone}
		\leq
		\norm{(\Id - \Ih) [\tbeta(\theta) \phih] }_{\Hone}
		\lesssim
		\norm{\tbeta(\theta)}_{\Wpolponeinf} \norm{\phih}_{\Ltwo},
	\end{align}
	where the last estimate can be obtained analogously to the proof of \cite[Lemma~5.2]{Doe24}. Thus, we have \eqref{est ut Rh}.
\end{proof}
\noindent Thanks to Lemma~\ref{lemma: est tbeta q terms}, we have
\begin{equation}
	\begin{aligned}
%		&
			\|\calFhu\|_{\LtwotLtwo}
%			\\
		\lesssim&\,%\begin{multlined}[t]
		C(\|u\|_{\Xu})h^\poldeg
		+	\norm{ \Delta_h \ehut }_{\LtwotLtwo}
		\bigl(	\norm{ \tbeta(\thetah)}_{\LinftLinf} + o(1) \bigr)\\ &
		+
		\norm{ \tbeta(\theta) - \tbeta(\thetah)}_{\LtwotWonedelta} \\ &
		+ 	\norm{ \Delta_h \ehu }_{\LtwotLtwo}
		\bigl(	\norm{ q(\thetah)}_{\LinftLinf} + o(1) \bigr)
		+
		\norm{ q(\theta) -  q(\thetah)}_{\LtwotWonedelta} \\ &
		+	\|\calN(\uh, \uht, \uhtt, \nabla \uh, \nabla \uht, \thetah)\\ &\qquad\qquad
		- \calN(\Rhu, \Rhut, \Rhutt, \nabla \Rhu, \nabla \Rhut, \theta)\|_{\LtwotLtwo}.
		%\end{multlined}
	\end{aligned}
\end{equation}

\begin{remark} \label{rem:L2_defect_u}
	One would expect that when measured in the $L^2(0,T; \Ltwo)$ norm, the defect scales as $h^{\eta + 1}$ as opposed to $h^\eta$ obtained above. 
	This reduction in the order of convergence has two theoretical sources.
	Firstly, due to the quadratic gradient nonlinearity in the Kuznetsov equations, in 
	Lemma~\ref{lemma: est defect u}, we need to estimate terms of the type $\nabla (u-\Rhu) \cdot \nabla \ut$,
	%, since they do not vanish as in the linear case, and 
	which are of order $\mathcal{O}(h^\eta)$.
	Secondly, the non-variational structure of the terms $q(\theta) \Delta u$
	and $\beta (\theta) \Delta \partial_t u$ do not allow us to
	choose the Ritz projection optimally, but only using
	to the zeroth-order terms of $q$ and $\beta$,
	which yields the defects that have to be handled in 
	Lemma~\ref{lemma: est tbeta q terms} with order $\mathcal{O}(h^\eta)$.
\end{remark}

In the next step, we bound the difference of the $\calN$ terms.

%%%%%%%%%%%%%%%%%%%%%%%%%%%%%%%%%%%%%%%%%%%%%%%%%%%%%%%%%%%%%%%%%%%%%%%%%%%%%%%%%%%%

\begin{lemma}\label{lemma: est calN}
	Under the assumptions of Theorem~\ref{thm main}, the following estimate holds on $[0, \finalth]$:
	\begin{equation} \label{est diff calN}
		\begin{aligned}
			&\|	\calN(\uh, \uht, \uhtt, \nabla \uh, \nabla \uht, \thetah)
%			\\ &
%			\qquad
			 -\calN(\Rhu, \Rhut, \Rhutt, \nabla \Rhu, \nabla \Rhut, \theta)\|_{\LtwotLtwo} \\
			\leq&\,\begin{multlined}[t] C(\norm{u}_{\Xu}, \norm{\theta}_{\Xtheta}) \Bigl\{ \|\kW(\theta)-\kW(\thetah))\|_{\LtwotLinf} + \|\kK(\theta)-\kK(\thetah))\|_{{\LtwotLinf}} \\
				+\|\kW(\thetah)\|_{\LinftLinf} \|\errhu\|_{\LtwotLinf}
				+ \|\errhut\|_{\LtwotLtwo}\\
				+\|\kK(\thetah)\|_{\LinftLinf} \|\errhut\|_{\LtwotLtwo} \\+	\norm{ \nabla \ehu}_{\LtwotLtwo}
				+
				\norm{ \nabla \ehtu}_{\LtwotLtwo}\Bigr\}\\
				+
				\Bigl\{
				\|\kW(\thetah) \, \uh\|_{\LinftLinf}+ \|\kK(\thetah) \, \uht\|_{\LinftLinf}
				\Bigr\}
				\|\errhutt\|_{\LtwotLtwo},
			\end{multlined}
		\end{aligned}
	\end{equation}
	where the constant is independent of $h$ and $\finalth$.
\end{lemma}
\begin{proof}
	We first rewrite the difference of $\calN$ terms as follows:
	\begin{equation} \label{diff calN}
		\begin{aligned}
			&\calN(\Rhu, \Rhut, \Rhutt, \nabla \Rhu, \nabla \Rhut, \theta)-\calN(\uh, \uht, \uhtt, \nabla \uh, \nabla \uht, \thetah)\\
			=&\, \begin{multlined}[t]
				( \kW(\theta)-\kW(\thetah))((\projRitz u)^2)_{tt} + \kW(\thetah)((\projRitz u)^2-\uh^2)_{tt}
				\\+ ( \kK(\theta)-\kW(\thetah))((\projRitz \ut)^2)_{t} + \kK(\thetah)((\projRitz \ut)^2-\uht^2)_{t}
				+2 \ell  \nabla \ehu \cdot \nabla \projRitz \ut
%				\\
					+
				2 \ell  \nabla \ph \cdot \nabla \ehtu .
			\end{multlined}
		\end{aligned}
	\end{equation}
	Using the fact that
	\[
	\begin{aligned}
		&((\projRitz u)^2-\uh^2)_{tt}
%		\\
		=
%&
		\,  2(\Rhu -\uh) \Rhutt+ 2\uh (\Rhutt - \uhtt)+2(\Rhut-\uht)(\Rhut+\uht),
	\end{aligned}
	\]
	we have by H\"older's inequality
	\begin{equation}
		\begin{aligned}
			&\norm{	( \kW(\theta)-\kW(\thetah))((\projRitz u)^2)_{tt} + \kW(\thetah)((\projRitz u)^2-\uh^2)_{tt} }_{\LtwotLtwo} \\
			\lesssim&\,\begin{multlined}[t]  	  \|\kW(\theta)-\kW(\thetah))\|_{\LtwotLinf}\|((\projRitz u)^2)_{tt}\|_{\LinftLtwo}
				\\
				+\|\kW(\thetah)\|_{\LinftLinf} \|\errhu\|_{\LtwotLinf}\|\Rhutt\|_{\LinfLtwo}
				\\ +
				\|\kW(\thetah) \, \uh\|_{\LinftLinf}  \|\errhutt\|_{\LtwotLtwo}
				\\
				+\|\errhut\|_{\LtwotLtwo}(\|\Rhut\|_{\LinfLinf}+\|\uht\|_{\LinftLinf})),
			\end{multlined}
		\end{aligned}
	\end{equation}
	we note that $\|\kW(\thetah)\|_{\LinftLinf} \lesssim 1$ thanks to Corollary~\ref{corollary: boundedness}. Next, since
	\[
	\begin{aligned}
		((\projRitz \ut)^2-\uht^2)_{t} = 2(\projRitz \ut-\uht) \Rhutt+\uht(\Rhutt - \utt)
	\end{aligned}
	\]
	we have
	\begin{equation}
		\begin{aligned}
			&\|( \kK(\theta)-\kK(\thetah))((\projRitz \ut)^2)_{t} + \kK(\thetah)((\projRitz \ut)^2-\uht^2)_{t} \|_{\LtwotLtwo} \\
			\lesssim&\,\begin{multlined}[t] 	\| \kK(\theta)-\kK(\thetah)\|_{\LtwotLinf}\|((\projRitz \ut)^2)_{t}\|_{\LinfLtwo}\\
				+\|\kK(\thetah)\|_{\LinftLinf} \|\nabla \errhut\|_{\LinftLtwo} \|\Rhutt\|_{\LtwoLthree}\\
				+	\|\kK(\thetah) \, \uht\|_{\LinftLinf}  \|\errhutt\|_{\LtwotLtwo}
			\end{multlined}
		\end{aligned}
	\end{equation}
	where we have relied on the embedding $\Hone \hookrightarrow \Lsix$. We can further employ the fact that $\|\kK(\thetah)\|_{\LinftLinf} \lesssim 1$. Lastly, we can estimate the gradient terms on the right-hand side of \eqref{diff calN} as follows:
	\begin{equation}
		\begin{aligned}
			& \|\ell  \nabla \ehu \cdot \nabla \projRitz \ut
			+
			\ell  \nabla \uh \cdot \nabla \ehtu \|_{\LtwotLtwo}\\
			\lesssim&\,
			\norm{ \nabla \ehu}_{\LtwotLtwo} \norm{\nabla \projRitz \ut}_{\LinfLinf}
			+
			\norm{\nabla \uh}_{\LinftLinf}
			\norm{ \nabla \ehtu}_{\LtwotLtwo},
		\end{aligned}
	\end{equation}
	where we recall that $	\norm{\nabla \uh}_{\LinftLinf} \lesssim 1$ on $[0, \finalth]$ thanks to Lemma~\ref{lemma: uniform bounds}. Combining the derived bounds leads to \eqref{est diff calN}.
\end{proof}
We have now all the ingredients to estimate $\calFhu$.
\begin{lemma} \label{lemma: est calFhu}
	Under the assumptions of Theorem~\ref{thm main}, we have
	\begin{equation}
		\begin{aligned}
%			&
				\|\calFhu\|_{\LtwotLtwo}
%				\\
			\leq&\,\begin{multlined}[t] C(\norm{u}_{\Xu}, \norm{\theta}_{\Xtheta}) \Bigl\{ h^{\poldeg}+\|\kW(\theta)-\kW(\thetah))\|_{\LtwotLinf}\\ + \|\kK(\theta)-\kK(\thetah))\|_{\LtwotLinf}
				\\
				+\|\kW(\thetah)\|_{\LinftLinf} \|\errhu\|_{\LtwotLinf}
				+\|\errhut\|_{\LtwotLtwo}\\ +\|\kK(\thetah)\|_{\LinftLinf} \|\errhut\|_{\LtwotLtwo} \\+	\norm{ \nabla \ehu}_{\LtwotLtwo}
				+	\norm{ \nabla \ehtu}_{\LtwotLtwo}+\|\Dh \errhu\|_{\LtwotLtwo} \Bigr\}\\
				+	\norm{ q(\theta) -  q(\thetah)}_{L^2_t(\Wonedelta)}+\norm{ \tbeta(\theta) -  \tbeta(\thetah)}_{L^2_t(\Wonedelta)}\\
				+	\norm{ \Delta_h \ehut }_{\LtwotLtwo}
				\bigl(	\norm{ \tbeta(\thetah)}_{\LinftLinf} + o(1) \bigr)\\
				+ \bigl(	\|\kW(\thetah)\, \uh\|_{\LinftLinf}
				+ \|\kK(\thetah) \, \uht\|_{\LinftLinf}
				\bigr) \|\errhutt\|_{\LtwotLtwo}.
			\end{multlined}
		\end{aligned}
	\end{equation}
\end{lemma}
\begin{proof}
	The estimate follows by combining the results of Lemmas~\ref{lemma: est defect u},~\ref{lemma: est tbeta q terms}, and~\ref{lemma: est calN}.
\end{proof}

\noindent By employing the maximal regularity estimate of Lemma~\ref{lemma: L2L2 estimate} with the choice $\wh=\delt \ehu$, we obtain
\begin{equation}
	\begin{aligned}
		&\|\ehutt\|_{\LtwotLtwo} + \|\Delta_h \ehut\|_{\LtwotLtwo}
		+
		\|\nabla \errhut\|_{\LinftLtwo}
%		\\
		\lesssim
%		 &
		  \,  \|\calFhu\|_{\LtwotLtwo}+\|\nabla \errhut(0)\|_{\Ltwo}.
	\end{aligned}
\end{equation}
Note that we can also bound $\norm{\Delta_h \errhu }_{\LinftLtwo} $ by using
\begin{equation}
	\norm{\Delta_h \errhu }_{\LinftLtwo}
	\lesssim
	\norm{\Delta_h \errhu(0)}_{\Ltwo}
	+
	\norm{  \partial_t \Delta_h \errhu }_{\LtwotLtwo}.
\end{equation}
Since by our choice of discrete initial data the terms at zero vanish, we have
\begin{equation} \label{max norm est wave}
	\begin{aligned}
		\begin{multlined}[t] \|\ehutt\|_{\LtwotLtwo} + \|\Delta_h \ehut\|_{\LtwotLtwo} +	\norm{\Delta_h \errhu }_{\LinftLtwo}
%			\\\hspace*{5.5cm}
 +
			\|\nabla \errhut\|_{\LinftLtwo}
			\lesssim\,  \|\calFhu\|_{\LtwotLtwo}. \end{multlined} %+\norm{\Delta_h \errhu(0)}_{\Ltwo} +
		%\|\errhut(0)\|_{\Hone}\\
	\end{aligned}
\end{equation}
This bound will be combined with an analogous one for the semi-discrete temperature equation, where we will look to either absorb the right-hand side terms by the left-hand side or handle them via \Gronwall's inequality.

\section{Estimates for the semi-discrete heat subproblem} \label{Sec: estimates semidiscrete heat}
In this section, we derive an estimate of $\errhtheta$ on $[0, \finalth]$ by suitably testing the semi-discrete heat subproblem. Recall that the heat equation in weak form is given by
\begin{align}
	(\thetat, \phi)_{L^2}
	+
	\kappa a(\theta, \phi) +\nu(\theta, \phi)_{L^2}
	=&\, (\alpha(\theta) (\zeta_1 u^2+\zeta_2 \ut^2), \phi)_{L^2}
\end{align}
for all $\phi \in \Honezero$. The semi-discrete version is then given by
\begin{align}
	(\thetaht, \phih)_{L^2}
	+
	\kappa a(\thetah, \phih) +\nu(\thetah, \phih)_{L^2}
	=\, (\alpha(\thetah) (\zeta_1 \uh^2+\zeta_2 (\uht)^2), \phih)_{L^2}
\end{align}
for all $\phih \in \Vh$. The Ritz projection of $\theta$ satisfies
\begin{align}
	&(\Rhthetat, \phih)_{L^2}+\kappa a(\Rhtheta, \phih) + \nu(\Rhtheta, \phih)_{L^2}
%	\\
%	=&
	\, (\alpha(\Rhtheta) (\zeta_1 (\Rhu)^2+\zeta_2 (\Rhut)^2), \phih)_{L^2}+	(\wtdeltatheta, \phih)_{L^2}
\end{align}
with the defect given by
\begin{equation} \label{eq:def_delta_theta}
	\begin{aligned}
		(\wtdeltatheta, \phih)_{L^2}  \coloneqq &\,\begin{multlined}[t] (\Rhthetat-\thetat, \phih)_{L^2}+ \nu(\Rhtheta-\theta, \phih)_{L^2}\\
			+ (\alpha(\theta) (\zeta_1 u^2+\zeta_2 \ut^2)-\alpha(\Rhtheta) (\zeta_1 (\Rhu)^2+\zeta_2 (\Rhut)^2), \phih)_{L^2}.
		\end{multlined}
	\end{aligned}
\end{equation}
Thus, the error $\ehtheta = \Rhtheta-\thetah$ solves the parabolic problem
\begin{equation} \label{eq ehtheta timediff}
	\begin{aligned}
		(\ehttheta, \phih)_{L^2}
		+
		\kappa  a(\ehtheta, \phih)+\nu(\ehtheta, \phih)_{L^2}
		=\, (\calFhtheta, \phih)_{L^2}
	\end{aligned}
\end{equation}
with the right-hand side
\begin{equation} \label{calFhTheta}
	\begin{aligned}
		(\calFhtheta, \phih)_{L^2}
		\coloneqq \begin{multlined}[t]
			(	\wtdeltatheta, \phih)_{L^2}
%			\\
			 + (\alpha(\Rhtheta) (\zeta_1 (\Rhu)^2+\zeta_2 (\Rhut)^2)-\alpha(\thetah) (\zeta_1 \uh^2+\zeta_2  (\uht)^2), \phih)_{L^2}.
		\end{multlined}
	\end{aligned}
\end{equation}
\indent Looking at the estimate of the acoustic right-hand side $\calFhu$ in Lemma~\ref{lemma: est calFhu}, we see that we have
% to be able
to further bound $\norm{ q(\theta) -  q(\thetah)}_{L^2_t(\Wonedelta)}$ and $\norm{ \tbeta(\theta) -  \tbeta(\thetah)}_{L^2_t(\Wonedelta)}$ in the course of the analysis of the semi-discrete heat equation. We intend to rely on the properties of the temperature-dependent speed of sound and sound diffusivity to conclude that
\begin{equation} \label{est h q beta}
	\| w(\theta) - w (\thetah)\|_{\LtwotWonedelta} \lesssim \|\theta -\thetah\|_{\LtwotWonedelta}, \quad w \in \{q, \tbeta\} \,,
\end{equation}
since $\|\theta\|_{\LinftWonedelta}$, $\|\thetah\|_{\LinftWonedelta} \leq C$. Similarly, we will exploit the estimate
\begin{equation} \label{est diff k terms}
	\begin{aligned}
		&\|\kW(\theta)-\kW(\thetah)\|_{\LtwotLinf} + \|\kK(\theta)-\kK(\thetah)\|_{\LtwotLinf}\\
		\lesssim&\, \|\theta-\thetah\|_{\LtwotLinf}\\
		\lesssim&\, \|\errhtheta\|_{\LtwotLinf}+\|\theta-\Rhtheta\|_{\LtwotLinf}
	\end{aligned}
\end{equation}
since $\|\theta\|_{\LinftLinf}$, $\|\thetah\|_{\LinftLinf} \leq C$. The error analysis of the heat equation should lead to bounds on $\|\errhtheta\|_{\LinftLinf}$ and $\|\errhtheta\|_{\LinftWonedelta}$ so that the terms $\|\errhtheta\|_{\LtwotLinf}$ and $\|\errhtheta\|_{\LtwotWonedelta}$ could be handled using \Gronwall's inequality. These can be obtained via the embedding \[\|\errhtheta\|_{\LinftLinf}+\|\errhtheta\|_{\LinftWonedelta} \lesssim \|\Dh \errhtheta\|_{\LinftLtwo}\] and a suitable bound on $\|\Dh \errhtheta\|_{\LinftLtwo}$.
\\
\indent To this end, we plan to employ Lemma~\ref{lemma: better estimate parabolic} on the semi-discrete heat subproblem, provided we have control of the right-hand side in the $H^1(0,t; \Ltwo)$ norm. We thus need to estimate $\|\calFhtheta\|_{\HonetLtwo}$ via bounds on
$\|\calFhtheta\|_{\LtwotLtwo}$ and $\| \delt \calFhtheta\|_{\LtwotLtwo}$.
We first estimate the defect term within $\calFhtheta$ using the usual approximation properties of the Ritz projection. To improve the readability, we postpone proofs of the next two lemmas to the Appendix.

\begin{lemma} \label{lemma: est defect theta}
	Under the assumptions of Theorem~\ref{thm main}, the following estimate holds:
	\begin{equation}
		\begin{aligned}
			\| \wtdeltatheta\|_{\LtwotLtwo}
			+
			\|\delt \wtdeltatheta\|_{\LtwotLtwo}
			\lesssim C(\|u\|_{\Xu}, \|\theta\|_{\Xtheta}) h^{\poldeg+1}.
		\end{aligned}
	\end{equation}
\end{lemma}

\begin{proof}
	The proof is given in Appendix~\ref{sec:appendix}.
\end{proof}

This result enables us to estimate $\calFhtheta$ and $\delt \calFhtheta$.
\begin{lemma} \label{lemma: est calFhtheta}
	Under the assumptions of Theorem~\ref{thm main}, we have the following two estimates on $[0, \finalth]$
	\begin{equation}
		\begin{aligned}
			\| \calFhtheta\|_{\LtwotLtwo} \leq&\, C(\|u\|_{\Xu}, \|\theta\|_{\Xtheta}) \begin{multlined}[t]\Bigl\{ h^{\poldeg+1}+\|\errhu\|_{\LtwotLtwo}
%				\\%\hspace*{2.5cm}
				+\|\errhut\|_{\LtwotLtwo} +
				\|\errhtheta\|_{\LtwotLtwo}
				\Bigr\},
			\end{multlined}
			\\
			\|\delt \calFhtheta\|_{\LtwotLtwo} \leq&\, C(\|u\|_{\Xu}, \|\theta\|_{\Xtheta}) \begin{multlined}[t]\Bigl\{h^{\poldeg+1}+\|\errhtheta\|_{\LtwotLtwo}+\|\errhthetat\|_{\LtwotLtwo}\\
				+ 	\|\Delta_h \errhtheta\|_{\LtwotLtwo}
				+\|\errhu\|_{\LtwotLtwo }\\ + \|\errhut\|_{\LtwotLtwo }
				+	\|\nabla \errhut \|_{\LtwotLtwo }
				\\
				+
				\|\alpha(\thetah) \,  \uht\|_{\LinftLinf} \|\errhutt\|_{\LtwotLtwo} \Bigr\} ,
			\end{multlined}
		\end{aligned}
	\end{equation}
	where the constants are independent of $h$ and $\finalth$.
\end{lemma}

\begin{proof}
	The proof is given in Appendix~\ref{sec:appendix}.
\end{proof}

From \eqref{eq ehtheta timediff} via parabolic estimate \eqref{est heat} with the choice $\wh=\ehtheta$, we have
\begin{equation} \label{max norm est heat}
	\begin{aligned}
		&\|\ehttheta\|_{\LinftLtwo}+\|\Delta_h \ehtheta\|_{\LinftLtwo}+			\|\nabla \errhtheta\|_{\LinftLtwo}
		\\
		\lesssim&\, \| \calFhtheta\|_{\HonetLtwo}+\|\Dh \errhtheta(0)\|_{\Ltwo}+\|\nabla \errhtheta(0)\|_{\Ltwo} = \| \calFhtheta\|_{\HonetLtwo}.
	\end{aligned}
\end{equation}
on $[0, \finalth]$. In the next section, we will combine this bound with the results of Section~\ref{Sec: estimate wave} to complete the proof Theorem~\ref{thm main}.

\section{Proof of the main result: \emph{A priori} bounds for the coupled system} \label{Sec: proof of the main result}
In this section, we  prolong the existence of $(\uh, \thetah)$ to $[0,T]$ and prove the main theoretical result of this work stated in Theorem~\ref{thm main}.
\subsection{Uniform bound for the wave-heat system} We first combine the two bounds for the semi-discrete pressure and heat subproblems to obtain the following result.
\begin{proposition} \label{prop:result_on_finalth}
	Let the assumptions of Theorem~\ref{thm main} hold.  The following estimate holds:
	\begin{equation} \label{uniform est}
		\begin{aligned}
			&\begin{multlined}[t]
				\|\ehttu\|_{\LtwotLtwo}
				+
				\|\Delta_h \ehut\|_{\LtwotLtwo}
				+
				\|\Delta_h \ehu\|_{\LinftLtwo}
				\\ +
				\| \errhut\|_{\LinftHone}
				+\|\Delta_h \ehtheta\|_{\LinftLtwo}
				+
				\| \errhthetat\|_{\LinftLtwo}
				\leq	C(\norm{u}_{\Xu}, \norm{\theta}_{\Xtheta})  h^{\poldeg}
			\end{multlined}
		\end{aligned}
	\end{equation}
	for $t \in [0, \finalth]$, where the constant is independent of $h$ and $\finalth$.
\end{proposition}
\begin{proof}
	Adding the two derived bounds \eqref{max norm est wave} and \eqref{max norm est heat} for the semi-discrete acoustic and heat subproblems leads to
	\begin{equation}
		\begin{aligned}
			&\begin{multlined}[t]
				\|\ehttu\|_{\LtwotLtwo}
				+
				\|\Delta_h \ehut\|_{\LtwotLtwo}
				+
				\|\Delta_h \ehu\|_{\LinftLtwo}
				+
				\| \errhut\|_{\LinftHone} \\
				+\|\Delta_h \ehtheta\|_{\LinftLtwo}
				+
				\| \errhthetat\|_{\LinftLtwo}+	\|\nabla \errhtheta\|_{\LinftLtwo}
			\end{multlined}  \\
			\lesssim&\, \begin{multlined}[t]
				\| \calFhu\|_{\LtwotLtwo}+ \| \calFhtheta\|_{\HonetLtwo}.
			\end{multlined}
		\end{aligned}
	\end{equation}
	On account of Lemmas~\ref{lemma: est calFhu} and~\ref{lemma: est calFhtheta}, we then find that
	\begin{equation} \label{interim est}
		\begin{aligned}
			&\begin{multlined}[t]
				\|\ehttu\|_{\LtwotLtwo}
				+
				\|\Delta_h \ehut\|_{\LtwotLtwo}
				+
				\|\Delta_h \ehu\|_{\LinftLtwo}
				\\
				+
				\|\nabla \errhut\|_{\LinftLtwo}
				+
				\|\Delta_h \ehtheta\|_{\LinftLtwo}
				+
				\| \errhthetat\|_{\LinftLtwo}
			\end{multlined}  \\
			\lesssim&\,	\begin{multlined}[t] C(\norm{u}_{\Xu}, \norm{\theta}_{\Xtheta}) \Bigl\{ h^{\poldeg}+\|\kW(\theta)-\kW(\thetah))\|_{\LtwotLinf}
%				 \\
				+ \|\kK(\theta)-\kK(\thetah))\|_{\LtwotLinf}
				+ \|\errhu\|_{\LtwotLinf}\\
				+\|\errhut\|_{\LtwotLtwo}+	\norm{ \nabla \ehu}_{\LtwotLtwo}
				+	\norm{ \nabla \ehtu}_{\LtwotLtwo}\\+\|\Dh \errhu\|_{\LtwotLtwo}+\|\errhtheta\|_{\LtwotLtwo}+\|\errhthetat\|_{\LtwotLtwo} \Bigr\}\\
				+\norm{ q(\theta) -  q(\thetah)}_{L^2_t(\Wonedelta)}+\norm{ \tbeta(\theta) -  \tbeta(\thetah)}_{L^2_t(\Wonedelta)}
				+ \calR,
			\end{multlined}
		\end{aligned}
	\end{equation}
	where we have introduced the following short-hand notation:
	\begin{equation}
		\begin{aligned}
			\calR \coloneqq &\,	\begin{multlined}[t]\norm{ \Delta_h \ehut }_{\LtwotLtwo}
				\bigl(	\norm{ \tbeta(\thetah)}_{\LinftLinf} + o(1) \bigr)
%				 \\
				+	\|\kW(\thetah) \ \uh\|_{\LinftLinf}
				\|\errhutt\|_{\LtwotLtwo}
				\\
				+ \|\kK(\thetah) \, \uht\|_{\LinftLinf} \|\errhutt\|_{\LtwotLtwo}
%				\\
				+	\|\alpha(\thetah) \, \uht\|_{\LinftLinf}\|\errhutt\|_{\LtwotLtwo}.
			\end{multlined}
		\end{aligned}
	\end{equation}
	We first have, using \eqref{est diff k terms},
	\begin{equation}
		\begin{aligned}
			&\|\kW(\theta)-\kW(\thetah))\|_{\LtwotLinf} + \|\kK(\theta)-\kK(\thetah))\|_{\LtwotLinf}
%			 \\
			%
			\lesssim
%			&\,
			 \|\Delta_h \errhtheta\|_{\LtwotLtwo}+\|\theta-\Rhtheta\|_{\LtwotLinf}.
		\end{aligned}
	\end{equation}
	The approximation properties of the Ritz projection stated in \eqref{approx properties projRitz} then yield
	\begin{equation}
		\begin{aligned}
			&\|\kW(\theta)-\kW(\thetah))\|_{\LtwotLinf}\| + \|\kK(\theta)-\kK(\thetah))\|_{\LtwotLinf} 
%			\\
			\lesssim
%			&\,
  \|\Delta_h \errhtheta\|_{\LtwotLtwo}+ h^{\poldeg} \|\theta\|_{L^2(W^{\poldeg, \infty}(\Omega))}.
		\end{aligned}
	\end{equation}
	The difference of $q$ and $\tbeta$ terms can be further estimated as follows:
	\begin{equation}
		\begin{aligned}
			&\norm{ q(\theta) -  q(\thetah)}_{\LtwotWonedelta}+\norm{ \tbeta(\theta) -  \tbeta(\thetah)}_{\LtwotWonedelta} \\
			\leq&\, C(\|\theta\|_{\LinftWdelta}, \|\thetah\|_{\LinftWdelta})\|\theta-\thetah\|_{\LtwotWdelta}\\
			\leq &\, C(\|\theta\|_{\LinftWdelta}, \|\thetah\|_{\LinftWdelta})\big(\|\errhtheta\|_{\LtwotWdelta}\\
			&\,\qquad +\|\theta-\Rhtheta\|_{\LtwotWdelta} \big),
		\end{aligned}
	\end{equation}
	and together with the discrete Sobolev embedding \eqref{discrete_Sobolev_embedding} (with $\dimension+\delta  \leq 6$) it holds
	\begin{equation}
		\|\errhtheta\|_{\LtwotWonedelta} \leq C \|\Delta_h \errhtheta\|_{\LtwotLtwo}.
	\end{equation}
	Again by the approximation properties of the Ritz projection stated in \eqref{approx properties projRitz}, we then have (with $\dimension+\delta  \geq 2$)
	\begin{equation}
		\begin{aligned}
%			&
			\norm{ q(\theta) -  q(\thetah)}_{\LtwotWonedelta}+\norm{ \tbeta(\theta) -  \tbeta(\thetah)}_{\LtwotWonedelta} 
%			\\
			\lesssim
%			&\,
			\|\errhtheta\|_{\LtwotWdelta}+h^{\poldeg}\|\theta\|_{L^2(W^{1+\poldeg, \dimension+\delta }(\Omega))},
		\end{aligned}
	\end{equation}
	where we have also used the fact that $\|\theta\|_{\LinftWdelta}$, $\|\thetah\|_{\LinftWdelta} \lesssim 1$. \\
	\indent It remains to discuss the terms within $\calR$. Observe that these terms cannot be handled using \Gronwall's inequality. Instead, we rely on the smallness of
	\begin{equation}
		\begin{aligned}
			& 
%			\norm{ \tbeta(\thetah)}_{\LinftLinf}
			+ \|\kW(\thetah) \uh\|_{\LinftLinf}
%			\\ &\,
			+
			\|\kK (\thetah) \uht\|_{\LinftLinf}
			+
			\|\alpha(\thetah) \uht\|_{\LinftLinf}
		\end{aligned}
	\end{equation}
	to absorb them by the left-hand side of \eqref{interim est}. This smallness can be achieved
	by the same computations as in \eqref{eq:thetat_theta_Linf} and  \eqref{eq:h_uh_Linf} for time $t$ instead of $0$,
	and performing the estimate \eqref{eq:k_u_products_Linf} also for
	$\tbeta(\thetah)$ and $\alpha(\thetah)\uht$.
	In fact, by the smallness condition
	\eqref{smallness condition u theta r}
	of the exact solution,
	we obtain
	\begin{equation}
		\begin{aligned}
%			&\,
			\norm{ \tbeta(\thetah)}_{\LinftLinf}
			+
			\|\kW(\thetah) \uh\|_{\LinftLinf}
%			\\ &\,
			+
			\|\kK (\thetah) \uht\|_{\LinftLinf}
			+
			\|\alpha(\thetah) \uht\|_{\LinftLinf}
			\leq r + C h_0^{\eps}.
		\end{aligned}
	\end{equation}
	Then by decreasing $r$ and $h_0$, the  $\errhu$ and $\errhtheta$ terms within $\calR$ can be absorbed by the left-hand side.  An application of \Gronwall's inequality thus yields
	\begin{equation}
		\begin{aligned}
			&\begin{multlined}[t]
				\|\ehttu\|_{\LtwotLtwo}
				+
				\|\Delta_h \ehut\|_{\LtwotLtwo}
				+
				\|\Delta_h \ehu\|_{\LinftLtwo}
				+
				\|\nabla \errhut\|_{\LinftLtwo}
				\\
				+\|\Delta_h \ehtheta\|_{\LinftLtwo}
				+
				\| \errhthetat\|_{\LinftLtwo}
				\leq	C(\norm{u}_{\Xu}, \norm{\theta}_{\Xtheta})  h^{\poldeg},
			\end{multlined}
		\end{aligned}
	\end{equation}
	as claimed.
\end{proof}
\subsection{Prolonging the interval of existence} We are now ready for the final step in the well-posedness and error analysis, which will complete the proof of the main theoretical result of this work.
\begin{proof}[Proof of Theorem~\ref{thm main}]
	Since the energy estimate \eqref{uniform est} holds on $[0, \finalth]$, we have
	\begin{equation}
		\begin{aligned}
			\|\Delta_h \ehu(\finalth)\|_{\Ltwo}
			+
			\| \errhut(\finalth)\|_{\Hone}
			&\leq
			C(\norm{u}_{\Xu}, \norm{\theta}_{\Xtheta})  h^{\poldeg}
			\\
			\|\Delta_h\ehtheta (\finalth)\|_{\Ltwo}
			%			+
			%			\| \errhthetat(\finalth)\|_{\Ltwo}
			&\leq	C(\norm{u}_{\Xu}, \norm{\theta}_{\Xtheta})  h^{\poldeg}.
		\end{aligned}
	\end{equation}
	On account of $\poldeg \geq 1$, we can then guarantee that
	\begin{equation}
		\begin{aligned}
			\|\Delta_h \ehu(\finalth)\|_{\Ltwo}
			+
			\| \errhut(\finalth)\|_{\Hone} &< C_0 h^{1/2+\eps},
			\\
			\|\Delta_h\ehtheta (\finalth)\|_{\Ltwo}
			%		+
			&< C_0 h^{1/2},
		\end{aligned}
	\end{equation}
	provided $h_0$ is sufficiently small.
	Therefore, by the same estimate for time $\finalth$ as in
	\eqref{eq:thetat_theta_Linf},  \eqref{eq:h_uh_Linf}, and \eqref{eq:k_u_products_Linf}
	$(\uh(\finalth), \uht(\finalth), \thetah(\finalth)) \in U_h$,
	where we recall that $U_h$ was defined in \eqref{def_Uh}. We can thus use the same reasoning from before but starting at the time $t=\finalth$ to prolong the existence of solutions beyond $\finalth$. The definition of $\finalth$ in \eqref{def:finaltimeh} then implies
	$\finalth=T$. Thus, the error estimate in \eqref{uniform est} holds on $[0,T]$. This completes the proof.
\end{proof}

\section{Numerical experiments} \label{Sec: Numerical experiments}

In this section, we explore fully discrete numerical approximations of the pressure-temperature  system \eqref{coupled_problem} and present numerical examples in two-dimensional spatial domains. For all the numerical examples of this section, we set $\Omega$ to be a bounded open domain of $\mathbb{R}^2$. We let $\{\mathcal{T}_h\}_{h>0}$ be a family of quasi-uniform triangulations of $\Omega$, which for a given meshsize $h>0$, every element $K\in \mathcal{T}_h$ corresponds to a triangle of diameter $h_K \leq h$. For the spatial discretization, we make use of the continuous Lagrangian finite element space $V_h$  of polynomial degree $\poldeg\geq 1$ introduced in \eqref{def:Vh}
for both the wave and heat equations.

To handle the nonlinearity arising in the functional $\mathcal{N}$ present in the wave equation in \eqref{coupled_problem}, we used a fixed-point iteration method. In addition, to treat the nonlinear $\theta$-dependent functions in the heat equation, we use a semi-implicit discretization in time. We let $\tau>0$ be the timestep and define $t^n := n\tau$ as the discrete time points for all $n\in \mathbb{N}$, and use the superscript notation $(\cdot)^n$ to denote time evaluations with $t=t^n$. Regarding the time discretizations, given a time-dependent function $a = a(t)$, we consider the backward Euler scheme:
\begin{align}\label{BDF1}
	\partial_\tau a^{n+1} = \dfrac{1}{\tau}\big(a^{n+1} - a^{n}\big)\quad\text{and}\quad
	\partial_\tau^2 a^{n+1} = \dfrac{1}{\tau^2}\big(a^{n+1}-2a^{n}+a^{n-1}\big),
\end{align}
and the second-order backward differentiation formulae (BDF2):
\begin{align}
	\begin{aligned}\label{BDF2}
		\partial_\tau a^{n+1} & = \dfrac{1}{\tau}\big(\tfrac{3}{2}a^{n+1} - 2a^{n} + \tfrac{1}{2}a^{n-1}\big),\\[1ex]
		\partial_\tau^2 a^{n+1} & =  \dfrac{1}{\tau^2}\big(2a^{n+1} - 5a^{n}+4a^{n-1} - a^{n-2}\big).
	\end{aligned}
\end{align}
We point out that in \cite{KovLL16}, it is shown that estimates as in Lemma~\ref{lemma: L2L2 estimate} carry over to time discretization with the implicit Euler and the BDF2 scheme.
To facilitate the writing of the fixed-point iteration method employed in the approximation of the wave equation, we define the following discrete operators
\begin{align}
	\delta_1 a^n:=  \begin{cases}
		a^{n}& \textup{Euler},\\[1mm]
		2a^{n} - \tfrac{1}{2}a^{n-1} & \textup{BDF2},\\[1mm]
	\end{cases}\qquad
	\delta_2 a^n:= \begin{cases}
		2a^{n} - a^{n-1}& \textup{Euler},\\[1mm]
		5a^{n} - 4 a^{n-1}  + a^{n-2}& \textup{BDF2},\\[1mm]
	\end{cases}
\end{align}
and the pair of constants $(\varsigma_1,\varsigma_2) = (1,1)$ for the implicit Euler method and $(\varsigma_1,\varsigma_2) = (\tfrac{3}{2}, 2)$ for BDF2. Note that $\varsigma_1$ and $\varsigma_2$ are the numbers multiplying $a^{n+1}$ in the first and second time derivatives, respectively, given for both time approximations. Then, with this notation, we can readily write
\begin{align} \label{eq:dtdiff}
	\partial_\tau a = \dfrac{1}{\tau}\big(\varsigma_1 a^{n+1} - \delta_1 a^n\big)\qquad \text{and}\qquad
	\partial_\tau^2 a = \dfrac{1}{\tau^2}\big(\varsigma_2 a^{n+1} - \delta_2 a^n\big).
\end{align}

\subsection*{Fully-discrete scheme} \label{subsec:semi-discrete} In order to perform the fixed-point iteration procedure, we first separate the terms containing the second-order time derivative of $u$ in \eqref{def calN}. For this purpose, we introduce two additional functionals $\calN_1$ and $\calN_2$, defined by
\begin{align}
	\begin{aligned}
		\calN_1(\theta,u,\ut) & \coloneqq \begin{cases}
			1+2\kW(\theta)u   &\textup{Westervelt}, \\[1mm]
			1+2\kK(\theta)u_t  &\textup{Kuznetsov},
		\end{cases}\\
		\calN_2(\theta,u_t, \nabla u, \nabla u_t) & \coloneqq \begin{cases}
			2\kW(\theta)u_t^2   &\, \textup{Westervelt}, \\[1mm]
			2\nabla u\cdot \nabla \ut  \quad  & \,\textup{Kuznetsov}.
		\end{cases}
	\end{aligned}
\end{align}
Observe that independent of the type of wave equation in \eqref{coupled_problem}, $\calN_1$ corresponds to the nonlinear coefficient of $\partial_t^2 u$, while $\calN_2$ contains the remaining non-linear terms of $\calN$.

Next, given $\uh^n,\uh^{n-1},\thetah^n\in V_{h}$,  and eventually $\uh^{n-2},\thetah^{n-1}\in V_{h}$ (cf.~\eqref{BDF2}), the fixed-point iteration procedure to update the wave equation from time $t^n$ to $t^{n+1}$, is established as follows. We define auxiliary variables $u_h^{(i)}\in V_{h}$ with $i\in\mathbb{N}$, and set $u_h^{(0)} = u_{h}^{n}$. Then, given $u_h^{(i)}$, the iterative process continues by finding $u_h^{(i+1)}\in V_{h}$ such that
\begin{align} \label{fully-discrete-wave}
	\begin{aligned}
		&\big(\calN_{1}^{(i)}\,\varsigma_2 u_{h}^{(i+1)}, \phih\big)_{L^2} + \tau^2 a\big(u_{h}^{(i+1)}, q(\thetah^n) \phih\big) +  \tau a\big(\varsigma_1 u_{h}^{(i+1)}, \beta (\thetah^{n}) \phih\big)\\
		&\quad = \big(\calN_{1}^{(i)}\,\delta_2 u_h^n, \phih\big)_{L^2} + \tau a\big(\delta_1 u_h^n, \beta (\thetah^{n}) \phih\big)
		-\tau^2\big(\calN_{2}^{(i)}-\fph^{n+1}, \phih\big)_{L^2}
		% 		+ \tau^2(\fph^{n+1}, \phih)_{L^2},
	\end{aligned}
\end{align}
for all $\phih \in V_{h}$, where the first and second time derivatives have been written using \eqref{eq:dtdiff} to  separate the variables arising in the iterative process from those of the discrete unknown at previous time steps, and
\begin{align}
	\calN_{1}^{(i)} & \coloneqq \, \calN_1\Big(\theta_h^n,u_{h}^{(i)},\tfrac{1}{\tau}\big(\varsigma_1 u_{h}^{(i)} - \delta_1 u_h^{n}\big) \Big),\\[1ex]
	\calN_{2}^{(i)} & \coloneqq \, \calN_2\Big(\theta^n_h,
	\tfrac{1}{\tau}\big(\varsigma_1 u_{h}^{(i)} - \delta_1 u_h^{n}\big), \nabla u_{h}^{(i)}, \tfrac{1}{\tau}\nabla\big(\varsigma_1 u_{h}^{(i)} - \delta_1 u_h^{n}\big)\Big).
\end{align}
The stopping criterion of the fixed-point iteration is set as follows:
\begin{equation}
	\dfrac{\|u_{h}^{(i+1)} - u_{h}^{(i)}\|_{\Ltwo }}{\|u_{h}^{(i+1)}\|_{\Ltwo}} < \texttt{tol},
\end{equation}
and the unknown is updated by setting $u_{h}^{n+1} = u_{h}^{(i+1)}$. For the heat equation, we discretize the equation in a semi-implicit fashion, and we solve for $\theta_h^{n+1}\in V_{h}$ such that
\begin{equation} \label{fully-discrete-heat}
	\begin{aligned}
		& (\varsigma_1\theta_h^{n+1}, \phih)_{L^2}
		+ \tau \kappa\, a(\thetah^{n+1}, \phih) + \tau \nu(\thetah^{n+1}, \phih)_{L^2}\\
		&\qquad\qquad =\, (\delta_1\theta_h^{n}, \phih)_{L^2} +  \tau \big(\calQ(\uh^{n+1}, \uth^{n+1}, \thetah^n), \phih\big)_{L^2},
	\end{aligned}
\end{equation}
for all $\phih\in V_{h}$. All numerical examples in this section have been implemented in Python using the open source finite element library FEniCSx \cite{barattaDOLFINxNextGeneration2023}. For the fixed-point iteration method, we use the tolerance $\texttt{tol} = 10^{-10}$. In addition, we remark that in order to properly initialize the BDF2 method, we perform the first time step with the implicit Euler method.
The codes to reproduce the results are available at
\[
\text{\mycode}
\]

\subsection{Example~1: Accuracy tests} \label{numex:1}
To determine the numerical errors produced by our numerical scheme \eqref{fully-discrete-wave}--\eqref{fully-discrete-heat} along with time discretizations \eqref{BDF1} or \eqref{BDF2}, we consider the unit square domain $\Omega = (0,1)^2$, and manufactured solutions $u_{\rm ex}\coloneqq u_{\rm ex}(x,t)$ and $\theta_{\rm ex}\coloneqq \theta_{\rm ex}(x,t)$. Then, the resulting terms arising after replacing $u=u_{\rm ex}$ and $\theta = \theta_{\rm ex}$ in \eqref{coupled_problem} are supplemented to the wave and heat equations through the respective source terms:
\begin{align}
	f_{\rm ex}(x,t)  \coloneqq &\,\,\partial_{t}^2 u_{\rm ex}-q(\theta_{\rm ex})\Delta u_{\rm ex} - \beta(\theta_{\rm ex}) \Delta ( \partial_t u_{\rm ex})\\
	& \, +\,  \calN\big(u_{\rm ex}, \partial_t u_{\rm ex}, \partial_{t}^2 u_{\rm ex}, \nabla u_{\rm ex}, \nabla (\partial_t u_{\rm ex}), \theta_{\rm ex}\big),\\[1mm]
	g_{\rm ex}(x,t)  \coloneqq &\,\, \partial_t\theta_{\rm ex} -\kappa\Delta \theta_{\rm ex}+ \nu\theta_{\rm ex} - \calQ(u_{\rm ex}, \partial_t u_{\rm ex}, \theta_{\rm ex}).
\end{align}
We note that the forcing term $g_{\rm ex}$ was originally not present in the second equation of \eqref{coupled_problem} and is introduced only for numerical testing. However, one could easily extend the error analysis above to this case, but we refrain from giving any details here.
Given $\lambda_j, A_j >0$ for $j \in\{1,2\}$, we define the following smooth manufactured solutions:
\begin{align}
	u_{\rm ex}(x_1,x_2,t)      & \coloneqq A_1 \sin(2 \pi x_1)\sin(2\pi x_2) {\rm exp}(\lambda_1 t), \label{def:uex}\\[1ex]
	\theta_{\rm ex}(x_1,x_2,t) & \coloneqq A_2 \sin(4 \pi x_1)\sin(4\pi x_2) {\rm exp}(-\lambda_2 t),\label{def:thetaex}%\\
\end{align}
which satisfy the zero Dirichlet boundary conditions ${u_{\rm ex}}_{|_{\partial\Omega}} = {\theta_{\rm ex}}_{|_{\partial\Omega}} = 0$. In addition, we set the discrete initial conditions as
\begin{align}
	u_{0,h} & = {\rm R}_h\, u_{\rm ex}(x_1,x_2,0),\quad %\\[1ex]%A_1 \sin(n_1 \pi x_1)\sin(n_1 \pi x_2),\\[1ex]
	u_{1,h}  = {\rm R}_h\, \partial_t u_{\rm ex}(x_1,x_2,0),\quad %\\[1ex]%A_2 \lambda_1\sin(n_1 \pi x_1)\sin(n_1 \pi x_2),\\[1ex]
	\theta_{0,h}  = {\rm R}_h \, \theta_{\rm ex}(x_1,x_2,0), %A_2 \sin(n_2 \pi x_1)\sin(n_2 \pi x_2),
\end{align}
for all $(x_1,x_2)\in \Omega$. For the error computations, we use a second order polynomial for the speed of sound $\sqrt{q}$ (truncated polynomial function for liver tissue in \cite{connor2002bio}) and respective sound diffusivity function: \begin{align}
%	c(\theta) & = 1529.3 + 1.6856\,(\theta+\Thetaa) + 6.1131\times 10^{-2}\, (\theta + \Theta_{\rm a})^2,\\
%\beta(\theta) & = \dfrac{2\tilde{\alpha}}{\omega^2}c^3(\theta),
	\sqrt{q(\theta)} & = 1529.3 + 1.6856\,(\theta+\Thetaa) + 6.1131\times 10^{-2}\, (\theta + \Theta_{\rm a})^2 \\
	\beta(\theta) & = \dfrac{2\tilde{\alpha}}{\omega^2}{(q(\theta))^{3/2}},
\end{align}
with $\Theta_{\rm a} = 37\, ^\circ{\rm C}$ being the ambient temperature, $\tilde{\alpha} = 4.5\times 10^{-6}\,\hat{f}\,{\rm Np \, m^{-1}}$ and $\omega = 2\pi \hat{f}$, for $\hat{f}=1\,\rm Hz$.
Furthermore, we set the coefficient functions $\kW$ and $\kK$ as
\begin{align}
%	\kW(\theta) =\dfrac{6}{\rho_{\rm a} c^2(\theta)},\qquad \kK(\theta)=\dfrac{5}{c^2(\theta)}, \label{eq:kW:kK:ex1}
	\kW(\theta) =\dfrac{6}{\rho_{\rm a} {q}(\theta)},\qquad \kK(\theta)=\dfrac{5}{{q}(\theta)}, \label{eq:kW:kK:ex1}
\end{align}
where $B/(2A) = 5$ (cf.~\eqref{nonlinearity coeffs}). For the heat equation, we set the constants $\kappa  = 1$ and $\nu= 10^{-5}$, the absorbed energy function
%\begin{align}\label{def:Q:example1}
%	\calQ(\theta,u,u_t) = \begin{cases}
%		\dfrac{1}{2\rho_{\rm a}}\left( \dfrac{\tilde{\alpha}}{ c(\theta)} u^2  +  \dfrac{2 b(\theta)}{c^4(\theta)}u_t^2\right) & \text{Westervelt},\\[2ex]
%		\dfrac{\rho_{\rm a} \tilde{\alpha}}{c(\theta)} u^2
%		& \text{Kuznetsov},
%	\end{cases}
%\end{align}
\begin{align}\label{def:Q:example1}
	\calQ(\theta,u,u_t) = \begin{cases}
		\dfrac{1}{2\rho_{\rm a}}\left( \dfrac{\tilde{\alpha}}{ {\sqrt{q(\theta)}}} u^2  +  \dfrac{2 b(\theta)}{{q^2}(\theta)}u_t^2\right) & \text{Westervelt},\\[2ex]
		\dfrac{\rho_{\rm a} \tilde{\alpha}}{{\sqrt{q(\theta)}}} u^2
		& \text{Kuznetsov},
	\end{cases}
\end{align}
with $\rho_{\rm a} = 1050\,\rm kg/m^3$. Given $\tau>0$ and $t=t^{n+1}$, we define the total error associated with our coupled numerical scheme at the time $t=t^{n+1}$ as follows
\begin{align} \label{def:total:error}
	\mathbf{E}_\tau(t^{n+1})  & =  \big\|\nabla \big( \delt u (t^{n+1})-\partial_\tau \uh^{n+1} \big) \big\|_{\Ltwo} + \big\| \delt\theta(t^{n+1})-\partial_{\tau}\thetah^{n+1}\big\|_{\Ltwo}\qquad\\
	& \quad  + \big\| \nabla (\theta(t^{n+1})-\thetah^{n+1})\big\|_{\Ltwo},
\end{align}
for all $n\geq 0$.
\begin{figure}[!t]
	\includegraphics[scale=0.425]{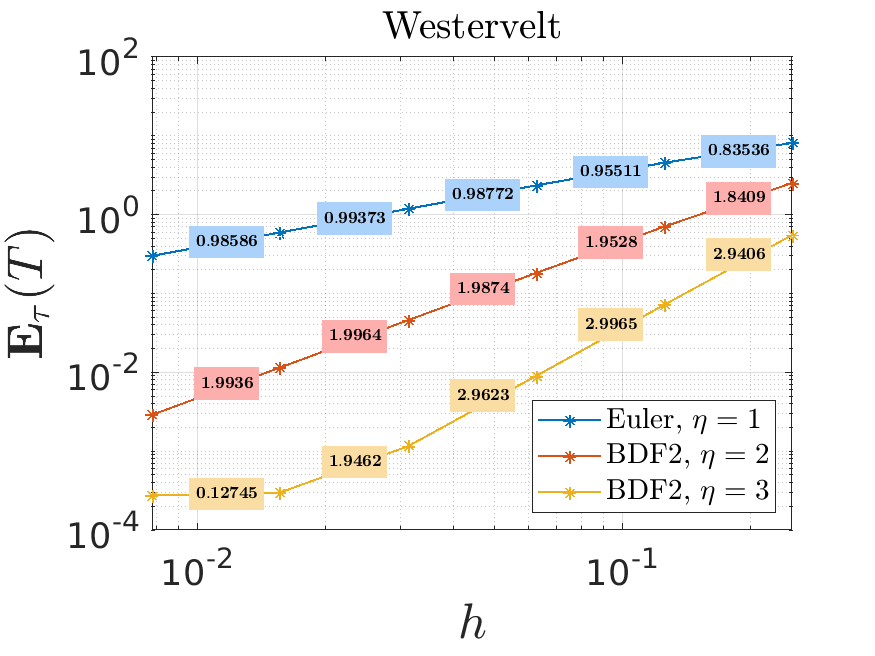}
	\includegraphics[scale=0.425]{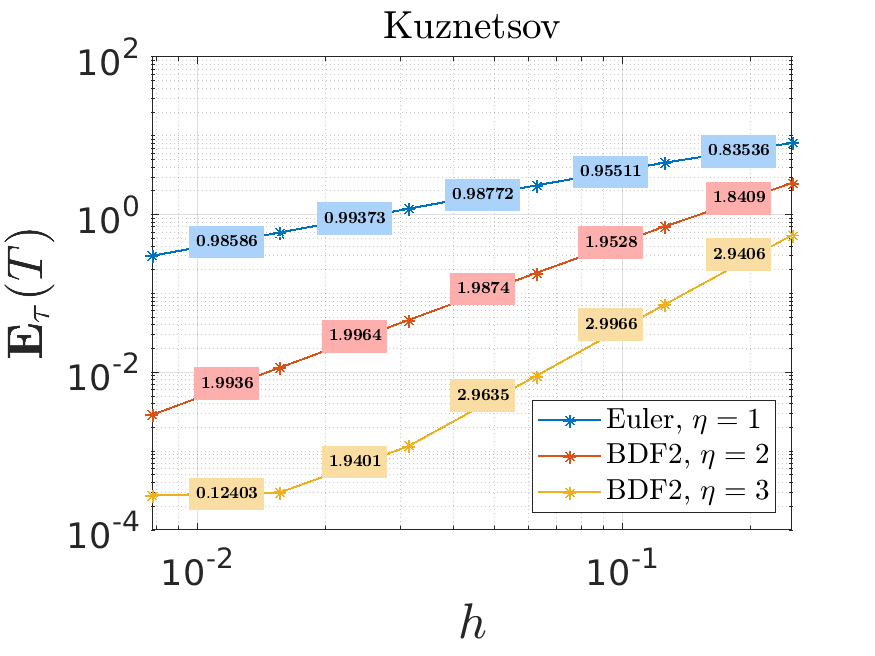}\\
	\includegraphics[scale=0.425]{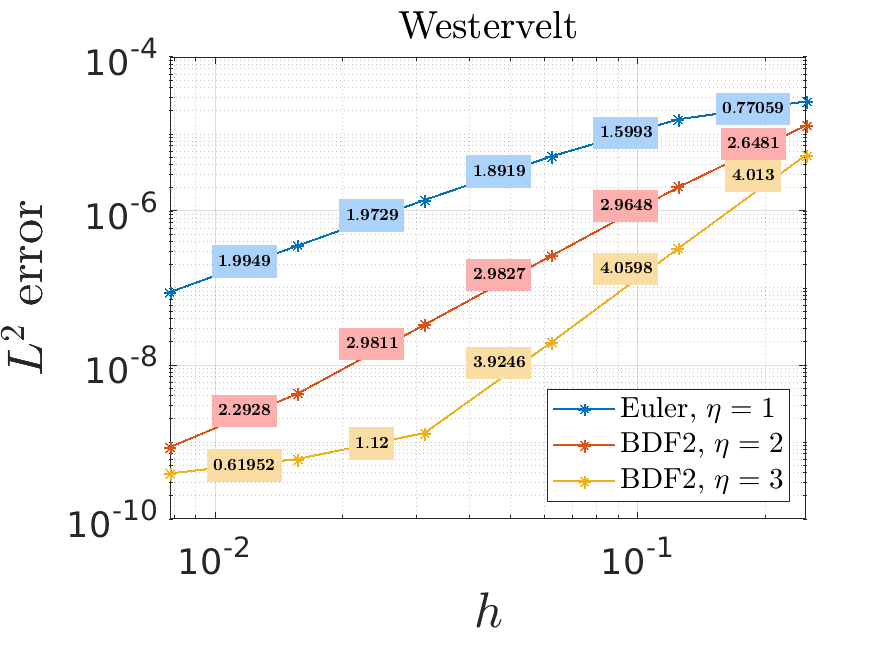}
	\includegraphics[scale=0.425]{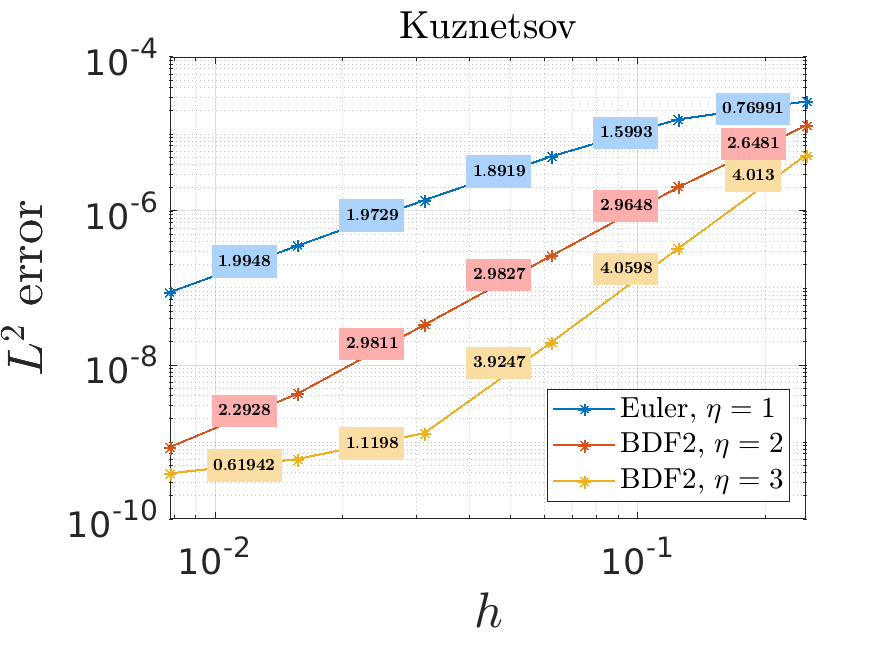}\\
	\caption{Total error \eqref{def:total:error} (first row) and $L^2$ error \eqref{def:total:error:Ltwo} (second row) computed from the numerical scheme making use of the implicit Euler method with $\poldeg=1$, and BDF2 method with $\poldeg=2$, and $\poldeg=3$, respectively. The errors are ploted against the meshsize $h$, for the common final time $T=1\,\rm s$ and timestep $\tau = 1/128\,{\rm s}=0.0078125\,{\rm s}$. The manufactured solutions are $u_{\rm ex}$ (cf.~\eqref{def:uex}) and $\theta_{\rm ex}$ (cf.~\eqref{def:thetaex}) with $A_1 = 1$, $A_2 = 10^{-4}$, $\lambda_1 = 1$, and $\lambda_2 = 1/2$.  \label{fig:errorplots1}}
\end{figure}
Figure~\ref{fig:errorplots1} shows the total error committed by our fully discrete numerical scheme for the two time approximations, implicit Euler and BDF2, with polynomial degrees $\poldeg = 1,2,3$, and for the Westervelt and Kuznetsov wave equations. For the two wave equations, the results show that in both cases, implicit Euler with $\poldeg=1$, and BDF2 with $\poldeg=2$, the orders of convergence are in agreement with what is predicted by Theorem~\ref{thm main}. For the case of BDF2 with $\poldeg=3$, we observe that the errors slopes tend to $\poldeg = 3$ until the order of convergence gets deteriorated due to the fact that $\tau$ overcomes $h$ in the error $\mathcal{O}(\tau^2+h^3)$. The later, explains the generation of a plateau as $h$ reaches the smallest values.

{
We further plotted in Figure~\ref{fig:errorplots1} the error in $L^2(\Omega)$ in space at time $T = N \tau$, i.e.
\begin{align} \label{def:total:error:Ltwo}
 \big\| u (T)- \uh^{N}  \big\|_{\Ltwo} + \big\| \theta(T)-\thetah^{N})\big\|_{\Ltwo},
\end{align}
 and observe the expected convergence of order $\eta +1 $. However, as explained in Remark~\ref{rem:L2_defect_u}, this is not covered by our theory.
}

\subsection{Example~2: Westervelt wave equation (initial excitation)} We let now $\Omega$ be composed by the union of the square region ${(-0.03, 0.03)\times(-0.04,0.04)}$ and the circular segment between $-\pi/4$ and $\pi/4$ of the circle centered at the origin with radius $0.05\,\rm m$. Therefore, $\Omega$ is not a polygonal domain and in this example (and in Example~3), we partially step aside from our theoretical framework. However, we still consider a polygonal approximation of the curved boundary of $\Omega$ to set the regular triangulation $\mathcal{T}_h$.
The domain, shown in both plots of Figure~\ref{fig:IC:source} is delimited by the black-dashed lines accounting for $\partial\Omega$. For this example, we simulate the Westervelt equation with manufactured initial condition described by the functions (in polar coordinates)
\begin{alignat}{2}
	\mathfrak{g}_0(r, \vartheta) \coloneqq 10^6  \cos\big(\tfrac{7}{4} \vartheta\big)\tfrac{3\pi}{r_0}r\exp\big(-\tfrac{3\pi}{r_0}r\big)\sin\big(\tfrac{15\pi}{r_0}r\big),\\ %&&\quad \vartheta \in [0,2\pi],\quad 0\leq r\leq r_0, \\
	\mathfrak{g}_1(r,\vartheta) \coloneqq 10^6  \cos\big(\tfrac{7}{4} \vartheta\big)\tfrac{3\pi}{r_0}r\exp\big(-\tfrac{3\pi}{r_0}r\big)\cos\big(\tfrac{15\pi}{r_0}r\big),
\end{alignat}
where $r$ is the radius measured from the origin and $\vartheta$ is the angle with respect to the $x_1$-axis, and $0.048\,{\rm m} = r_0 < 0.05\,{\rm m}$. Then, at $t=0$, we set $u$ and $\partial_t u$ as follows
\begin{align} \label{eq:IC:example2}
	\big(u_0(r,\vartheta),u_1(r,\vartheta)\big) =
	\begin{cases}
		(\mathfrak{g}_0(r,\vartheta), \mathfrak{g}_1(r,\vartheta)) & \text{ if }
		-\tfrac{2\pi}{7}\leq \vartheta \leq \tfrac{2\pi}{7},\quad 0\leq r\leq r_0,\\[1ex]
		(0,0) &\text{ otherwise in $\Omega$},
	\end{cases}
\end{align}
and the zero initial temperature $\theta_0\equiv 0$. Figure~\ref{fig:IC:source}a displays the plot of $u_0$ as a function of $(x_1,x_2)$ in $\Omega$. We remark that the initial conditions are built in order to satisfy the zero Dirichlet boundary conditions and being continuous functions within the domain. The initial wave amplitude is maximal at the $x_1$-axis, when $\vartheta = 0$ (see the continuous red line in Figure~\ref{fig:IC:source}) and it decreases to zero towards the lines $\vartheta = \pm 2\pi/7$ and $r=0$.

For the coefficient functions in \eqref{coupled_problem}, we set the temperature-dependent speed of sound {$\sqrt{q}$} as the fifth order polynomial function modeling liver tissue in \cite{connor2002bio}
\begin{align}
%	c(\theta) & 
{\sqrt{q(\theta)}} & = 1529.3 + 1.6856\,(\theta+\Theta_{\rm a}) + 6.1131\times 10^{-2} \,(\theta+\Theta_{\rm a})^2\\
	& \quad - 2.2967 \times 10^{-3}\, (\theta+\Theta_{\rm a})^3 + 2.2657 \times 10^{-5}\,(\theta+\Theta_{\rm a})^4 \\
	& \quad - 7.1795 \times 10^{-8}\, (\theta+\Theta_{\rm a})^5.
\end{align}
In addition, the corresponding sound diffusivity $\beta$ and $\kW$ function are taken as in Example~1 after setting the frequency $\hat{f} = 100\,{\rm k Hz}$. No source term is included in the wave equation in this example, i.e., $f\equiv 0$.
The parameters used in the heat equation, which are related to liver tissue as the ambient 'a', and blood 'b' are taken from \cite[Table~3]{connor2002bio} to be
\begin{align}
	&\beta_{\rm a} = 6\,{\rm kg^3\,m^{-4}\,s^{-2}}, %
	\quad \rho_{\rm a} =1050\, {\rm kg\,m^{-3}}, \quad \rho_{\rm b} = 1030\, {\rm kg\,m^{-3}},\quad \Theta_{\rm a} = 37\, ^\circ{\rm C} \\
	&C_{\rm a}  = 3600\, {\rm J\,kg\, K^{-1}},\quad C_{\rm b} = 3620\,{\rm J\,kg\, K^{-1}},\quad \kappa_{\rm a} = 0.512\, {\rm W\,m^{-1}\,K^{-1}},
\end{align}
and the constant diffusion $\kappa$ and parameter $\nu$ are respectively given by
\begin{align}
	\kappa & =  \dfrac{\kappa_{\rm a}}{\rho_{\rm a}C_{\rm a}},\qquad\nu = \dfrac{\rho_{\rm b}C_{\rm b}}{\rho_{\rm a}C_{\rm a}}.
\end{align}
For the absorbed acoustic energy functional, we use the definition given in \eqref{def:Q:example1}.

\begin{figure}[t]
	\begin{tabular}{cc}
		\bf (a) & \bf  (b)\\
		\includegraphics[scale=0.45]{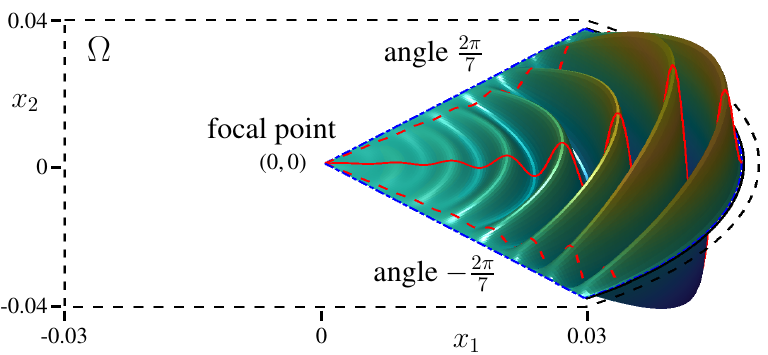} &
		\includegraphics[scale=0.45]{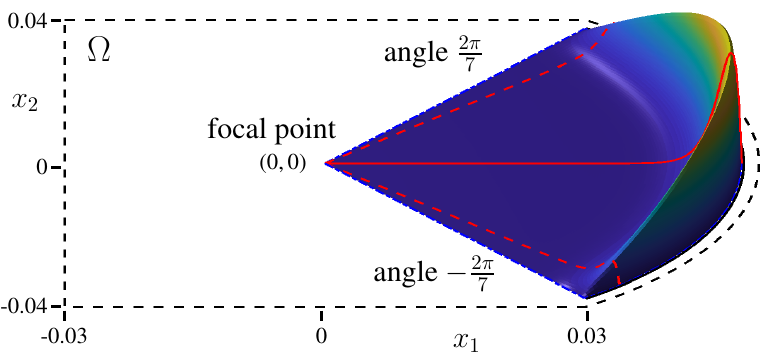}
	\end{tabular}
	\caption{ {\bf (a)} Initial pressure $u_0=u_0(x_1,x_2)$ used in Example~2, and {\bf (b)} source function $f = f(x_1,x_2,t)$ at $t=0$ employed in Example~3. The continuous red lines correspond to the respective functions at $x_2 = 0$ for $0\leq x_1\leq r_0$, and the red-dashed lines stand for the functions' plot at the angles $-\pi/4$ and $\pi/4$ respectively.
		The black-dashed contour coincide with the boundary $\partial \Omega$.\label{fig:IC:source}}
\end{figure}

For the numerical simulations, we use the BDF2 scheme with $\tau = 10^{-7}\,\rm s$ together with a polynomial approximation of degree $\poldeg = 1$, linear piecewise polynomials, and set $T= 4\times10^{-5}\,\rm s$. The meshsize is set to $h=7.6\times 10^{-4}\,\rm m$ and the number of elements of the used mesh is 38912. In Figure~\ref{fig:ex2:pressure}, we show the simulated pressure profiles $u_h^n$ at $n = 10,50,100,200$ and $300$. The combined effect of having the initial time derivative $u_1$, and $u_0$ given by \eqref{eq:IC:example2} is that the wave amplitude reaches its maximum towards the focal point instead of directly dissipating, and travels towards the left boundary. Furthermore, it can be clearly seen that initially the acoustic wave traveling in the direction of the $x_1$-axis gets reflected from the curved boundary. The discrete temperature, on the other hand, is presented in Figure~\ref{fig:ex2:temp}, in which it can be observed the heating effect that the ultrasound wave has on the focal area, where the temperature reaches its maximum.

\begin{figure}
	\centering
	\includegraphics[scale=0.76]{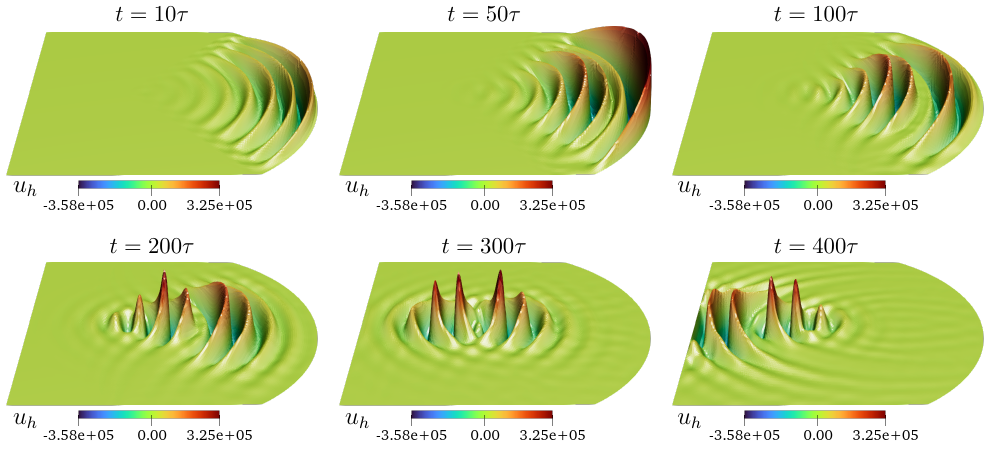}
	\caption{Example~2: Snapshots of the discrete pressure $u_h = u_h(x_1,x_2)$ computed with the Westervelt wave equation in \eqref{coupled_problem}, initial coditions \eqref{eq:IC:example2} and $\theta_0 \equiv 0$, and setting the source term $f = 0$. The time step is $\tau = 10^{-7}\,\rm s$ and $\poldeg = 1$, and the time approximation is by the BDF2 method. \label{fig:ex2:pressure}}
\end{figure}

\begin{figure}[t!]
	\centering
	\includegraphics[scale=0.76]{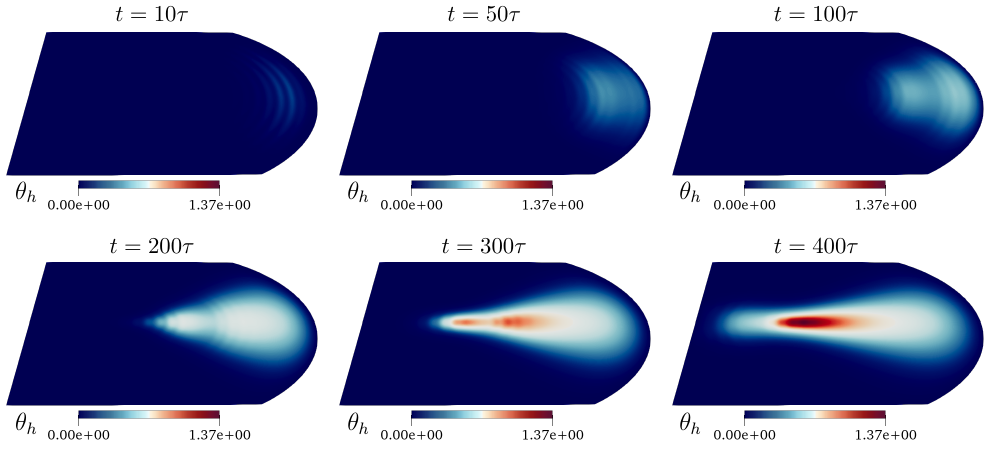}
	\caption{Example~2: Snapshots of the discrete temperature $\theta_h = \theta_h(x_1,x_2)$ computed with the Westervelt wave equation in \eqref{coupled_problem}, initial coditions \eqref{eq:IC:example2} and $\theta_0 \equiv 0$, and setting the source term $f = 0$. The time step is $\tau = 10^{-7}\,\rm s$ and $\poldeg = 1$, and the time approximation is by the BDF2 method. \label{fig:ex2:temp}}
\end{figure}

\subsection{Example~3: Kuznetsov wave equation (source excitation)} In this example, we explore the case of a high frequency excitation due to a source term on the wave equation for the Kuznetsov equation. Unlike Example~2, the unknown $u$ represents now the acoustic velocity potential. Under the same conditions of Example~2, meaning the same domain, parameters and coefficient functions (with $\kK$ as in \eqref{eq:kW:kK:ex1}), we set the source term function similarly as the initial conditions in Example~2 through the function (in polar coordinates)
\begin{align} \label{eq:source:ex3}
	\mathfrak{f}_0(r, \vartheta,t) \coloneqq 10^8  \cos\big(\tfrac{7}{4} \vartheta\big)
	\tfrac{1}{r_0}r\Big(\exp\big(-40 r/r_0\big) - \exp\big(-40\big) \Big)\cos(\omega t),
\end{align}
where the pair $(r,\vartheta)$ corresponds to the polar coordinates relative to $(x_1,x_2)\in \Omega$, and $r_0=0.048\,\rm m$. Then, we define the source term function as follows:
\begin{align}
	f(r,\vartheta,t) \coloneqq \begin{cases}
		\mathfrak{f}_0(r, \vartheta,t) &\text{ if }-\tfrac{2\pi}{7}\leq \vartheta \leq \tfrac{2\pi}{7},\quad 0\leq r\leq r_0,\\[1ex]
		0 &\text{ otherwise in $\Omega$}.
	\end{cases}
\end{align}
Figure~\ref{fig:IC:source}b) shows the plot of function $f$ at the time $t=0$. Varying time $t$, the described source term oscillates with angular frequency $\omega = 2\pi \hat{f}$, with $\hat{f} = 100\,\rm kHz$. This function is intended to mimic the effect of having an excitation due to Neumann boundary conditions, but keeping the unknowns to zero at the boundaries. To perform this numerical example, we use the BDF2 method with $\tau = 10^{-7}\,{\rm s}$, $T= 4\times10^{-5}\,\rm s$, linear elements with $\poldeg = 1$, and the same mesh as in Example~1.

In Figure~\ref{fig:pressure:temp:ex3}, we show snapshots of the discrete solution $u_h$ (top row), and discrete temperature (bottom row) at three times $t^n$ with $n = 200,300$ and $400$. Unlike the previous example, in which the wave is induced by the initial condition, in this case the oscillations of the source term continuously drive the wave. Then, sequential peaks of $u_h$ reach the focal area as time evolves. The approximated temperature $\theta_h$ is presented in Figure~\ref{fig:pressure:temp:ex3}, which is in the order of magnitude of $10^{-9}\rm \,^\circ C$.
The later can be explained due to the reduced magnitude of $u_h$, which directly influences the strength of the absorbed energy function $\calQ$.

\begin{figure}[t!]
	\centering
	\includegraphics[scale=0.76]{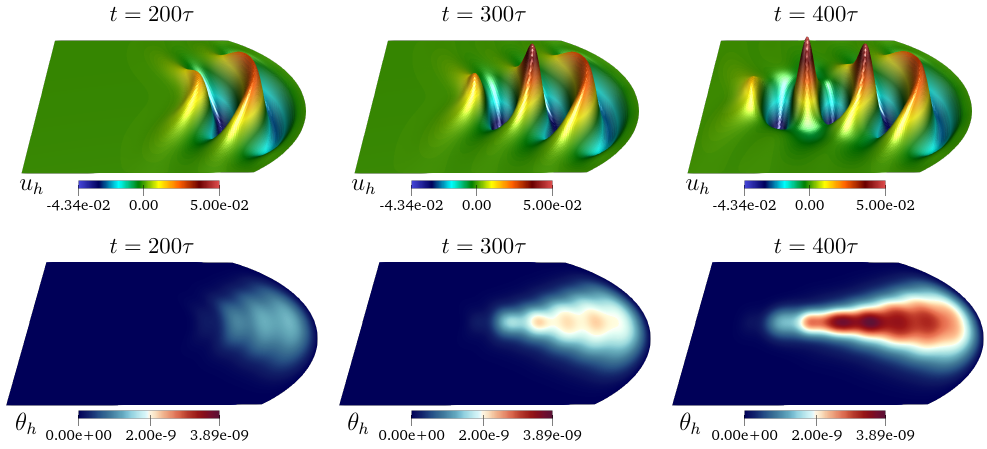}
	\caption{Example~3: Snapshots of the discrete acoustic velocity potential $u_h = u_h(x_1,x_2)$ (first row) and discrete temperature $\theta_h = \theta_h(x_1,x_2)$ (second row), at three time points, computed with the Kuznetsov wave equation in \eqref{coupled_problem} with zero initial conditions and source term $f$ given by \eqref{eq:source:ex3}. The time step is $\tau = 10^{-7}\,\rm s$ and $\poldeg = 1$, and the time approximation is by the BDF2 method. \label{fig:pressure:temp:ex3}}
\end{figure}

\appendix

% \section{Appendices}

\section{Postponed proofs}
\label{sec:appendix}

In this section, we finally give the proofs from Sections~\ref{Sec: Ritz proj} and \ref{Sec: estimates semidiscrete heat} on the bound of the Ritz projection and the defects from the error analysis of the heat part.

\begin{proof}[Proof of Lemma~\ref{Lemma: est Ritz product_v2}]
	We add and subtract $\mu(\psi_h) \, \phih$ to obtain
	\[
	\norm{\projRitz [ \mu(\psi_h) \, \phih ]}_{\Ltwo} \leq \norm{ \mu(\psi_h) \, \phih}_{\Ltwo}+\norm{(\Id - \projRitz)  [ \mu(\psi_h) \, \phih ]}_{\Ltwo},
	\]
	which yields
	\[
	\norm{\projRitz  [ \mu(\psi_h) \, \phih ]}_{\Ltwo} \leq 	\norm{\mu(\psi_h)}_{\Linf} \norm{ \phih }_{\Ltwo} +\norm{(\Id - \projRitz)  [ \mu(\psi_h) \, \phih ] }_{\Ltwo}.
	\]
	To estimate the second term on the right-hand side, we note that
	
	{by an Aubin-Nitsche trick}	%
	\begin{equation}
		\norm{(\Id - \projRitz)  [ \mu(\psi_h) \, \phih ] }_{\Ltwo}
		\lesssim
		h \norm{(\Id - \projRitz) [ \mu(\psi_h) \, \phih ] }_{\Hone}
		\lesssim
		h \norm{(\Id - \Ih) [ \mu(\psi_h) \, \phih ] }_{\Hone}.
	\end{equation}
	Using the standard interpolation estimate, we obtain
	\begin{align}
		h^2 \norm{(\Id - \Ih) (\mu(\psi_h) \, \phih)}_{\Hone}^2
		&\leq C
		\sum_K \bigl( h^{\poldeg+1} | \mu(\psi_h) \, \phih |_{H^{\poldeg+1}(K)} \bigr)^2 .
		%		\\
	\end{align}
	We use that on each cell $\phih$ is a polynomial of degree $\poldeg$, together with the inverse estimate in  \cite[Lemma (4.5.3)]{brenner2008mathematical},
	to derive
	\begin{align}
		h^{\poldeg+1} | \mu(\psi_h) \, \phih |_{H^{\poldeg+1}(K)}
		& \lesssim
		h^{\poldeg+1}  \sum_{j=1}^{\poldeg+1}
		| \mu(\psi_h)  |_{W^{j,\infty}(K)}
		|  \phih |_{H^{\poldeg+1-j}(K)}
%		\\&
		\lesssim
	{	\|  \phih \|_{L^2(K)}}
		\sum_{j=1}^{\poldeg+1}
		h^{j}
		| \mu(\psi_h)  |_{W^{j,\infty}(K)} .
	\end{align}
	The observation that
	\begin{align}
		h^{j} | \mu(\psi_h)  |_{W^{j,\infty}(K)}
		\lesssim
		\sum_{\ell=1}^{j} h^\ell \norm{\nabla \psi_h}_{L^\infty(K)}^\ell
	\end{align}
	allows us to conclude
	\begin{align}
		h^2 \norm{(\Id - \Ih) (\mu(\psi_h) \, \phih)}_{\Hone}^2
		&\leq C
		\sum_K
		\norm{  \phih }_{L^2(K)}^2
		\sum_{j=1}^{\poldeg+1}
		h^{2j} \norm{\nabla \psi_h}_{L^\infty(K)}^{2j} 
%		\\ &
		%
		\leq
		C	\norm{\phih}_{\Ltwo}^2
		\sum_{j=1}^{\poldeg+1}
		\bigl( h \norm{\nabla \psi_h}_{L^\infty(\Omega)}\bigr)^{2j}.
	\end{align}
	Finally, we employ
	\begin{equation}
		h  \norm{\nabla \psi_h }_{\Linfty(\Omega)} \lesssim 	h^{\delta / (\dimension+\delta )} \norm{\nabla \psi_h }_{\Wonedelta}
	\end{equation}
	to obtain the claim.
\end{proof}

Next, we turn to the defect defined in \eqref{eq:def_delta_theta}.

\begin{proof}[Proof of Lemma~\ref{lemma: est defect theta}]
	(a)	We have the rewriting
	\begin{equation}
		\begin{aligned}
%			& 
			(\wtdeltatheta, \phih)_{L^2}
%			\\
			=
%			&\,
			\begin{multlined}[t]
				(\Rhthetat-\thetat, \phih)_{L^2}+ \nu(\Rhtheta-\theta, \phih)_{L^2}
				+ ((\alpha(\theta)-\alpha(\Rhtheta)) (\zeta_1 u^2+\zeta_2 \ut^2), \phih)_{L^2}\\
				+( \alpha(\Rhtheta) (\zeta_1 (u-\Rhu)(u+\Rhu) + \zeta_2 (\ut-\Rhut)(\ut+\Rhut)), \phih)_{L^2}%\\ &
				% 				+( \alpha(\Rhtheta) (\zeta_2 (\ut-\Rhut)(\ut+\Rhut)), \phih)_{L^2}.
			\end{multlined}
		\end{aligned}
	\end{equation}
	The statement then follows by the approximation properties of the Ritz projection.
	
	(b) For the time derivative of the defect, we note that
	\begin{equation}
		\begin{aligned}
%			&
			(\delt \wtdeltatheta, \phih)_{L^2}
%			\\
			=
%			&\,
			\begin{multlined}[t] (\Rhthetatt-\thetatt, \phih)_{L^2}+ \nu(\Rhthetat-\thetat, \phih)_{L^2}\\
				+ (\alpha'(\theta)\thetat (\zeta_1 u^2+\zeta_2 \ut^2)-\alpha'(\Rhtheta)\Rhthetat (\zeta_1 (\Rhu)^2+\zeta_2 (\Rhut)^2), \phih)_{L^2}\\
				- 2(\alpha(\theta) \ut (\zeta_1 u +\zeta_2 \utt)-\alpha(\Rhtheta) \Rhut (\zeta_1 \Rhu +\zeta_2 \Rhutt), \phih)_{L^2}.
			\end{multlined}
		\end{aligned}
	\end{equation}
	We have the following rewriting:
	\begin{equation} \label{term1}
		\begin{aligned}
			& (\alpha'(\theta)\thetat (\zeta_1 u^2+\zeta_2 \ut^2)-\alpha'(\Rhtheta)\Rhthetat (\zeta_1 (\Rhu)^2+\zeta_2 (\Rhut)^2), \phih)_{L^2}\\
			=&\,\begin{multlined}[t] 	((\alpha'(\theta)-\alpha'(\Rhtheta))\thetat (\zeta_1 u^2+\zeta_2 \ut^2), \phih)_{L^2}\\
				+(\alpha'(\Rhtheta)(\thetat-\Rhthetat)(\zeta_1 (\Rhu)^2+\zeta_2 (\Rhut)^2), \phih)_{L^2}\\
				+\big(\alpha'(\Rhtheta)\Rhthetat(\zeta_1 (\uh-\Rhu)(\uh+\Rhu)\\ +\zeta_2 (\uht-\Rhut)(\uht+\Rhut)), \phih\big)_{L^2}.
			\end{multlined}
		\end{aligned}
	\end{equation}
	On account of the local Lipschitz continuity of $\alpha'$, we have
	\begin{equation}
		\begin{aligned}
			&\|(\alpha'(\theta)-\alpha'(\Rhtheta))\thetat (\zeta_1 u^2+\zeta_2 \ut^2) \|_{\LtwoLtwo}\\
			\lesssim &\,\|\alpha'(\theta)-\alpha'(\Rhtheta)\|_{\LinfLtwo}\|\thetat\|_{\LtwoLinf}(\|u\|^2_{\LinfLinf}+\|\ut\|^2_{\LinfLinf}) \\
			\lesssim&\, \|\theta-\Rhtheta\|_{\LinfLtwo}\|\thetat\|_{\LtwoLinf}(\|u\|^2_{\LinfLinf}+\|\ut\|^2_{\LinfLinf})
		\end{aligned}
	\end{equation}
	and we can estimate the other terms on the right-hand side of \eqref{term1} in an analogous manner. Similarly, we have the rewriting
	\begin{equation}
		\begin{aligned}
			&(\alpha(\theta) \ut (\zeta_1 u +\zeta_2 \utt)-\alpha(\Rhtheta) \Rhut (\zeta_1 \Rhu +\zeta_2 \Rhutt), \phih)_{L^2} \\
			=&\,\begin{multlined}[t]
				((\alpha(\theta)-\alpha(\Rhtheta)) \ut (\zeta_1 u +\zeta_2 \utt), \phih)_{L^2}
%				 \\
				+(\alpha(\Rhtheta)(\ut-\Rhut) (\zeta_1 u +\zeta_2 \utt), \phih)_{L^2}
				\\
				+(\alpha(\Rhtheta)\Rhut (\zeta_1 (u-\Rhu) +\zeta_2 (\utt-\Rhutt)), \phih)_{L^2}
			\end{multlined}
		\end{aligned}
	\end{equation}
	and we can proceed as above to arrive at the claim.
\end{proof}

The last estimate deals with the right-hand side in the error equation of the heat problem.

\begin{proof}[Proof of Lemma~\ref{lemma: est calFhtheta}]
	We use the following rewriting:
	\begin{equation} \label{calFhTheta}
		\begin{aligned}
%			& 
			(\calFhthetat, \phih)_{L^2} 
%			\\
			=
%			 & \,
			  \begin{multlined}[t]
				(\wtdeltatheta, \phih)_{L^2} + ((\alpha(\Rhtheta)-\alpha(\thetah)) (\zeta_1 (\Rhu)^2+\zeta_2 (\Rhut)^2)), \phih)_{L^2}\\ +(\alpha(\thetah) (\zeta_1 (\Rhu-\uh)(\Rhu+\uh)+\zeta_2 (\Rhut-\uht)(\Rhut+\uht)), \phih)_{L^2}.
			\end{multlined}
		\end{aligned}
	\end{equation}
	We further have
	\begin{equation}
		\begin{aligned}
			&	\|	(\alpha(\Rhtheta)-\alpha(\thetah)) (\zeta_1 (\Rhu)^2+\zeta_2 (\Rhut)^2))\|_{\LtwotLtwo}\\
			\lesssim&\, \|\errhtheta\|_{\LtwotLtwo}\|\zeta_1 (\Rhu)^2+\zeta_2 (\Rhut)^2\|_{\LinftLinf},
		\end{aligned}
	\end{equation}
	where we have relied on
	\[
	\|\alpha(\Rhtheta)-\alpha(\thetah)\|_{\LtwotLtwo} \lesssim \|\Rhtheta-\thetah\|_{\LtwotLtwo}
	\]
	for $\|\thetah\|_{\LinftLinf} \lesssim 1$. Similarly,
	\begin{equation}
		\begin{aligned}
			&\|\alpha(\thetah) (\zeta_1 (\Rhu-\uh)(\Rhu+\uh)+\zeta_2 (\Rhut-\uht)(\Rhut+\uht))\|_{\LtwotLtwo}\\
			\lesssim&\, \|\alpha(\thetah)\|_{\LinftLinf}(\|\errhu\|_{\LtwotLtwo}+\|\errhut\|_{\LtwotLtwo})\big(\|\uh\|_{\HonetLtwo}
%			\\
%			&\hspace{3cm}
			+\|\Rhu\|_{\HonetLtwo}\big).
		\end{aligned}
	\end{equation}
	The claim then follows by Lemma~\ref{lemma: est defect theta} and the properties of the Ritz projection.
	
	We now tackle the estimate of $\delt \calFhtheta$.
	Note that
	\begin{equation} \label{calFhTheta timediff}
		\begin{aligned}
			(\delt \calFhtheta, \phih)_{L^2}
			=\, \begin{multlined}[t]
				(\delt	\wtdeltatheta, \phih)_{L^2} + \calI,
			\end{multlined}
		\end{aligned}
	\end{equation}
	where
	\begin{equation} \label{calFhTheta timediff}
		\begin{aligned}
			\calI
			\coloneqq &\, \begin{multlined}[t]
				\big(\alpha'(\Rhtheta)\Rhthetat (\zeta_1 (\Rhu)^2
%				\\
				 +\zeta_2 (\Rhut)^2)+2\alpha(\Rhtheta) \Rhut (\zeta_1 \Rhu +\zeta_2 \Rhutt), \phih\big)_{L^2}\\
				- (\alpha'(\thetah)\thetaht (\zeta_1 \uh^2+\zeta_2 (\uht)^2)+2\alpha(\thetah) \uht (\zeta_1 \uh +\zeta_2 \uhtt), \phih)_{L^2}.
			\end{multlined}
		\end{aligned}
	\end{equation}
	We can rewrite $\calI$ as follows:
	\begin{equation}
		\begin{aligned}
			\calI
			=&\, \begin{multlined}[t]
				(\alpha'(\Rhtheta)\Rhthetat (\zeta_1 (\Rhu)^2+\zeta_2 (\Rhut)^2)-\alpha'(\thetah)\thetaht (\zeta_1 \uh^2+\zeta_2 (\uht)^2), \phih)_{L^2} \\
				+2(\alpha(\Rhtheta) \Rhut (\zeta_1 \Rhu +\zeta_2 \Rhutt)-\alpha(\thetah) \uht (\zeta_1 \uh +\zeta_2 \uhtt), \phih)_{L^2}.
			\end{multlined}
		\end{aligned}
	\end{equation}
	We next further rewrite the two difference terms. First,
	\begin{equation}
		\begin{aligned}
			&(\alpha'(\Rhtheta)\Rhthetat (\zeta_1 (\Rhu)^2+\zeta_2 (\Rhut)^2)-\alpha'(\thetah)\thetaht (\zeta_1 \uh^2+\zeta_2 (\uht)^2), \phih)_{L^2} \\
			=&\, \begin{multlined}[t] ((\alpha'(\Rhtheta)-\alpha'(\thetah))\Rhthetat (\zeta_1 (\Rhu)^2+\zeta_2 (\Rhut)^2), \phih)_{L^2}		\\
				+ (\alpha'(\thetah)(\Rhthetat-\thetaht)  (\zeta_1 (\Rhu)^2+\zeta_2 (\Rhut)^2), \phih)_{L^2}	\\
				+\zeta_1 (\alpha'(\thetah)\thetaht  (\Rhu-\uh)(\Rhu + \uh), \phih)_{L^2}	\\
				+ \zeta_2(\alpha'(\thetah)\thetaht  (\Rhut-\uht)(\Rhut+\uht)), \phih)_{L^2}
				:=\, \sum_{i=1}^{4} (\calI^i, \phih)_{L^2}.				\end{multlined}
		\end{aligned}
	\end{equation}
	We have
	\begin{equation}
		\begin{aligned}
%			&
			\|	\calI^1\|_{\LtwotLtwo}
%			\\
%			=
			&\,
			 \|(\alpha'(\Rhtheta)-\alpha'(\thetah))\Rhthetat (\zeta_1 (\Rhu)^2+\zeta_2 (\Rhut)^2)\|_{\LtwotLtwo}\\
			\lesssim&\, \|\alpha'(\Rhtheta)-\alpha'(\thetah)\|_{\LtwotLtwo}\|\Rhthetat\|_{\LinfLinf} \Big(\|\Rhu\|_{\LinfLinf}^2\\
			&\hspace{3cm}+\|\Rhut\|^2_{\LinfLinf}\Big)\\
			\lesssim&\, \|\theta\|_{\Xtheta} \|u\|^2_{\Xu}
			\|\errhtheta\|_{\LtwotLtwo}	.
		\end{aligned}
	\end{equation}
	Secondly,
	\begin{equation}
		\begin{aligned}
			& \|	\calI^2\|_{\LtwotLtwo}= \|\alpha'(\thetah)(\Rhthetat-\thetaht)  (\zeta_1 (\Rhu)^2+\zeta_2 (\Rhut)^2)\|_{\LtwotLtwo}\\
			\lesssim&\,\|\alpha'(\thetah)\|_{\LinftLinf}\|\errhthetat\|_{\LtwotLtwo} \Big(\|\Rhu\|_{\LinfLinf}^2+\|\Rhut\|^2_{\LinfLinf}\Big).
		\end{aligned}
	\end{equation}
	Thirdly,
	\begin{equation}
		\begin{aligned}
			\|\calI^3\|_{\LtwotLtwo}
			=&\, \|\zeta_1 \alpha'(\thetah)\thetaht  (\Rhu-\uh)(\Rhu + \uh)\|_{\LtwotLtwo} \\
			%	\lesssim&\, \|\alpha'(\thetah)\|_{\LinftLinf}\|\thetaht\|_{\LinftLinf}  \|\errhu\|_{\LtwotLtwo }(\|\Rhu\|_{\LinfLinf} + \|\uh\|_{\LinftLinf}).
			%		\\
			\lesssim&\,	 \|\thetaht\|_{\LinftLp{3}}
			\|\errhu\|_{\LtwotLp{6} }
			\lesssim
			\|\thetaht\|_{\LinftLp{3}}
			\|\nabla \errhu \|_{\LtwotLtwo }	.
		\end{aligned}
	\end{equation}
	Next,
	\begin{equation}
		\begin{aligned}
			 \|\calI^4\|_{\LtwotLtwo}
			&=\|\zeta_2 \alpha'(\thetah)\thetaht  (\Rhut-\uht)(\Rhut+\uht)\|_{\LtwotLtwo}\\
			%							\\
			&\lesssim\, \|\thetaht\|_{\LinftLp{3}}
			\|\errhut\|_{\LtwotLp{6} }
			\lesssim
			\|\thetaht\|_{\LinftLp{3}}
			\|\nabla \errhut \|_{\LtwotLtwo }.
		\end{aligned}
	\end{equation}
	Similarly, we have the following rewriting:
	\begin{equation}
		\begin{aligned}
			&	2(\alpha(\Rhtheta) \Rhut (\zeta_1 \Rhu +\zeta_2 \Rhutt)-\alpha(\thetah) \uht (\zeta_1 \uh +\zeta_2 \uhtt), \phih)_{L^2} \\
			=&\, \begin{multlined}[t]
				2((\alpha(\Rhtheta)-\alpha(\thetah)) \Rhut (\zeta_1 \Rhu +\zeta_2 \Rhutt), \phih)_{L^2} \\
				+	2(\alpha(\thetah) (\Rhut-\uht) (\zeta_1 \Rhu +\zeta_2 \Rhutt), \phih)_{L^2}\\
				+	2(\alpha(\thetah)\uht (\zeta_1 (\Rhu-\uh) +\zeta_2 (\Rhutt-\uhtt)), \phih)_{L^2}
			\end{multlined}	\\
			:=&\, \sum_{i=5}^{7} (\calI^i, \phih)_{L^2}.
		\end{aligned}
	\end{equation}
	Then
	\begin{equation}
		\begin{aligned}
			&\|\calI^5\|_{\LtwotLtwo}
			= \|2(\alpha(\Rhtheta)-\alpha(\thetah)) \Rhut (\zeta_1 \Rhu +\zeta_2 \Rhutt)\|_{\LtwotLtwo} \\
			\lesssim&\, \|\errhtheta\|_{\LtwotLinf}
			\|\Rhut\|_{\LinftLinf}(\|\Rhu\|_{\LtwoLtwo}+\|\Rhutt\|_{\LtwoLtwo}).
		\end{aligned}
	\end{equation}
	Next,
	\begin{equation}
		\begin{aligned}
			&\|\calI^6\|_{\LtwotLtwo}
			= \|2\alpha(\thetah) (\Rhut-\uht) (\zeta_1 \Rhu +\zeta_2 \Rhutt)\|_{\LtwotLtwo} \\
			\lesssim&\,\|\alpha(\thetah)\|_{\LinftLinf} \|\errhthetat\|_{\LinftLsix}(\|\Rhu\|_{\LtwoLthree}+\|\Rhutt\|_{\LtwoLthree}).
		\end{aligned}
	\end{equation}
	Finally,
	\begin{equation}
		\begin{aligned}
			&\|\calI^7\|_{\LtwotLtwo}
			= \|	2\alpha(\thetah)\uht (\zeta_1 (\Rhu-\uh) +\zeta_2 (\Rhutt-\uhtt))\|_{\LtwotLtwo} \\
			\lesssim&\,
			\|\alpha(\thetah) \, \uht\|_{\LinftLinf}(\|\errhu\|_{\LtwotLtwo}
			+
			\|\errhutt\|_{\LtwotLtwo}).
		\end{aligned}
	\end{equation}
	Combining the derived bounds yields the desired result.
\end{proof}

%%-----------------------------
%%      your bibliography
%%-----------------------------

\section*{Acknowledgments} \label{Sec: funding}
{\bf Funding.} J.C. is supported by ANID through Fondecyt project 3230553.
B.D. is funded by the Deutsche Forschungsgemeinschaft (DFG, German Research Foundation) -- Project-ID 258734477 -- SFB 1173. The work of V.N.  was partially supported by the  Dutch Research Council (NWO) under the grant OCENW.M.23.371.
%

% \section*{References}

\bibliography{references}{}
\bibliographystyle{amsplain}

\end{document}